\documentclass{amsart}%
\usepackage{color}
\usepackage{amsfonts}
\usepackage{amsmath}
\usepackage{amsthm}
\usepackage{amssymb}
\usepackage{graphicx}%
\providecommand{\U}[1]{\protect\rule{.1in}{.1in}}
\DeclareMathOperator{\dex}{dex}
\newtheorem{theorem}{Theorem}[section]
\newtheorem{lemma}[theorem]{Lemma}
\newtheorem{corollary}[theorem]{Corollary}
\theoremstyle{definition}
\newtheorem{definition}[theorem]{Definition}

\theoremstyle{remark}
\newtheorem{remark}[theorem]{Remark}
\numberwithin{equation}{section}
\begin{document}
\title[Bivariate Shepard-Bernoulli operators]{Bivariate Shepard-Bernoulli operators}
\author[F. Dell'Accio]{F. Dell'Accio}
\address{Dipartimento di Matematica e Informatica, Universit\`{a} della Calabria, 87036 Rende (Cs), Italy}
\email{fdellacc@unical.it}
\author[F. Di Tommaso]{F. Di Tommaso}
\address{Dipartimento di Matematica e Informatica, Universit\`{a} della Calabria, 87036 Rende (Cs), Italy}
\email{ditommaso@mat.unical.it}
\date{}
\dedicatory{To Professor F.A. Costabile}

\begin{abstract}
In this paper we extend the Shepard-Bernoulli operators introduced in \cite{CaiDel} to the bivariate case. These new
interpolation operators are realized by using local support basis functions
introduced in \cite{FraNie} instead of classical Shepard basis functions and
the bivariate three point extension \cite{CosDel2}\ of the generalized Taylor
polynomial introduced \ by F. Costabile in \cite{Cos}. The new operators do
not require either the use of special partitions of the node convex hull or
special structured data as in \cite{Cat}. We deeply study their approximation
properties and provide an application to the scattered data interpolation
problem; the numerical results show that this new approach is comparable with
the other well known bivariate schemes QSHEP2D and CSHEP2D by Renka
\cite{Ren2,Ren4}.

\end{abstract}

\maketitle

\section{The problem}

Let $\mathcal{N}=\left\{  \boldsymbol{x}_{1},...,\boldsymbol{x}_{N}\right\}  $
be a set of $N$ distinct points (called \textit{nodes} or \textit{sample
points}) of $\mathbb{R}^{s},s\in\mathbb{N}$, and let $f$ be a function defined
on a domain $D$ containing $\mathcal{N}$. The classical Shepard operators
(first introduced in \cite{She} in the particular case $s=2$) are defined by%
\begin{equation}
S_{N,\mu}\left[  f\right]  \left(  \boldsymbol{x}\right)  :={\displaystyle\sum
\limits_{i=1}^{N}}A_{\mu,i}\left(  \boldsymbol{x}\right)  f\left(
\boldsymbol{x}_{i}\right)  ,\quad\mu>0, \label{shepoper}%
\end{equation}
where the weight functions $A_{\mu,i}\left(  \boldsymbol{x}\right)  $ in
barycentric form are%
\begin{equation}
A_{\mu,i}\left(  \boldsymbol{x}\right)  :=\frac{\left\vert \boldsymbol{x}%
-\boldsymbol{x}_{i}\right\vert ^{-\mu}}{{\displaystyle\sum\limits_{k=1}^{N}%
}\left\vert \boldsymbol{x}-\boldsymbol{x}_{k}\right\vert ^{-\mu}}
\label{sheweighfun}%
\end{equation}
and $\left\vert \cdot\right\vert $ denotes the Euclidean norm in
$\mathbb{R}^{s}$. The interpolation operator $S_{N,\mu}\left[  \cdot\right]  $
is stable, in the sense that%
\[
\min_{i}f\left(  \boldsymbol{x}_{i}\right)  \leq S_{N,\mu}\left[  f\right]
\left(  \boldsymbol{x}\right)  \leq\max_{i}f\left(  \boldsymbol{x}_{i}\right)
\]
but for $\mu>1$ the interpolating function $S_{N,\mu}\left[  f\right]  \left(
\boldsymbol{x}\right)  $ has flat spots in the neighborhood of all data
points. Moreover, the degree of exactness of the operator $S_{N,\mu}\left[
\cdot\right]  $\ is $0$, i.e. if it is restricted to the polynomial space
$\mathcal{P}_{\boldsymbol{x}}^{m}:=\left\{  p:\deg\left(  p\right)  \leq
m\right\}  $, then $\left.  S_{N,\mu}\left[  \cdot\right]  \right\vert
_{\mathcal{P}_{\boldsymbol{x}}^{m}}=Id_{\mathcal{P}_{\boldsymbol{x}}^{m}}$
(the identity function on $\mathcal{P}_{\boldsymbol{x}}^{m}$) only for $m=0$.

These drawbacks, in particular, can be avoided by replacing each value
$f\left(  \boldsymbol{x}_{i}\right)  $ in (\ref{shepoper}) with an
interpolation operator in $\boldsymbol{x}_{i}$, applied to $f$, with a degree
of exactness $m>0$ \cite{CaiDel}. In addition, to make the Shepard method a
local one, according to \cite{FraNie} we multiply Euclidean distances
$\left\vert \boldsymbol{x}-\boldsymbol{x}_{j}\right\vert ,j=1,\ldots,N$ in
(\ref{sheweighfun}) by Franke-Little weights \cite{HosLas}
\[
\left(  1-\frac{\left\vert \boldsymbol{x}-\boldsymbol{x}_{j}\right\vert
}{R_{w_{j}}}\right)  _{+}^{\mu},\text{ }R_{w_{j}}>0\text{ for each }%
j=1,\ldots,N
\]
where $(\cdot)_{+}$ is the positive part of the argument. As a result,
functions $A_{\mu,i}\left(  \boldsymbol{x}\right)  $ are replaced by compact
support functions $\widetilde{W}_{\mu,i}\left(  \boldsymbol{x}\right)  ,$
$i=1,\ldots,N$ defined as follows:
\begin{equation}
\widetilde{W}_{\mu,i}\left(  \boldsymbol{x}\right)  :=\dfrac{W_{\mu,i}\left(
\boldsymbol{x}\right)  }{\sum\limits_{k=1}^{N}W_{\mu,k}\left(  \boldsymbol{x}%
\right)  } \label{Wtbasisfunctions}%
\end{equation}
where%
\[
W_{\mu,i}\left(  \boldsymbol{x}\right)  :=\left(  \frac{1}{\left\vert
\boldsymbol{x}-\boldsymbol{x}_{i}\right\vert }-\frac{1}{R_{w_{i}}}\right)
_{+}^{\mu},
\]
and $R_{w_{i}}$ is the radius of influence about node $\boldsymbol{x}_{i}$: in
practice $R_{w_{i}}$ is taken to be just large enough to include $N_{w}$ nodes
in the open ball $B\left(  \boldsymbol{x}_{i},R_{w_{i}}\right)  =\left\{
\boldsymbol{x}\in%
%TCIMACRO{\U{211d} }%
%BeginExpansion
\mathbb{R}
%EndExpansion
^{s}:\left\vert \boldsymbol{x}-\boldsymbol{x}_{i}\right\vert <R_{w_{i}%
}\right\}  $ \cite{Ren2}.

More precisely, if $P\left[  \cdot,\boldsymbol{x}_{i}\right]  $ denotes an
interpolation operator in $\boldsymbol{x}_{i}$ based on data in the ball
$B\left(  \boldsymbol{x}_{i},R_{w_{i}}\right)  $, $i=1,\dots,N$, the
\textit{local combined Shepard operators} are then defined as follows:
\begin{equation}
\widetilde{S}_{N,\mu}P[f](\boldsymbol{x}):={\displaystyle\sum\limits_{i=1}%
^{N}}\widetilde{W}_{\mu,i}\left(  \boldsymbol{x}\right)  P[f,\boldsymbol{x}%
_{i}]\left(  \boldsymbol{x}\right)  . \label{improvshep}%
\end{equation}
We emphasize that weight functions $\widetilde{W}_{\mu,i}\left(
\boldsymbol{x}\right)  $ have the same basic properties as weight functions
$A_{\mu,i}\left(  \boldsymbol{x}\right)  $, but the value of the improved
Shepard operator (\ref{improvshep}) at a point $\boldsymbol{x}\in D$ is
influenced only by values $\widetilde{W}_{\mu,i}\left(  \boldsymbol{x}\right)
P[f,\boldsymbol{x}_{i}]\left(  \boldsymbol{x}\right)  $ with $\boldsymbol{x}%
_{i}\in\mathcal{N}_{\boldsymbol{x}}:=\left\{  \boldsymbol{x}_{k}\in
\mathcal{N}:\left\vert \boldsymbol{x}-\boldsymbol{x}_{k}\right\vert <R_{w_{k}%
}\right\}  $.

Starting from Shepard himself various combinations have been proposed in order
to improve the approximation properties of the Shepard operators (see, for
example, \cite{ComTri1,Far}): Shepard-Bernoulli operators introduced recently
in \cite{CaiDel} represent a further attempt. These univariate operators are
obtained by replacing each value $f\left(  x_{i}\right)  $ in (\ref{shepoper})
with the generalized Taylor polynomial $P_{m}[f,x_{i},x_{i+1}](x)$, proposed
by F. Costabile in \cite{Cos}. In \cite{CaiDel} the rate of convergence of the
Shepard-Bernoulli operators is studied in depth and numerical examples
demonstrate the accuracy of the proposed combination in special situations, in
particular, when it is applied to the problem of interpolating the discrete
solutions of initial value problems for ordinary differential equations. In
the conclusion of \cite{CaiDel} the possibility of extending the
Shepard-Bernoulli operator to higher dimension was hypothesized by using the
expansions studied in \cite{CosDel1,CosDel2}.

In 2007 T. Catinas \cite{Cat} combined classic Shepard operators with the
tensorial extension of the generalized Taylor polynomial discussed in
\cite{CosDel1}. The resulting combination has separated degree of exactness
$m$ with respect to $x$ and $n$ with respect to $y$ when applied to
sufficiently smooth functions in the convex hull of data; it uses $N+2\left(
N-1\right)  $ specially structured three-dimensional data (in the situation
depicted in \cite{Cat} each point of $\mathcal{N}$ is the vertex of a
rectangle with vertices in $\mathcal{N}$) but, on account of the nature of the
polynomial, it interpolates only $N$ of them. The numerical results provided
in the paper, on some of the test functions provided in \cite{RenCli,RenBro},
show that the accuracy of the operator can be improved by using compact
support basis functions $\widetilde{W}_{\mu,i}$ instead of the global basis
functions $A_{\mu,i}$.

In this paper we extend the Shepard-Bernoulli operators to the bivariate case
using local basis functions $\widetilde{W}_{\mu,i}$ and the three point
interpolation polynomials discussed in \cite{CosDel2} and introduce a new
combination which interpolates on all data used for its definition. We do this
by associating to each sample point $\boldsymbol{x}_{i}$ a triangle with a
vertex in it and other two vertices in $B\left(  \boldsymbol{x}_{i},R_{w_{i}%
}\right)  $; the association is done in order to reduce the error of the three
point interpolation polynomial based on the three vertices of the triangle.
For fixed values of $N_{w}$ \cite{Ren2} this choice allows us to reduce the
error of the proposed combination. As a consequence, the resulting operator
not only interpolates at each sample point and increases by $1$ the degree of
exactness of the Shepard-Taylor operator \cite{Far} which uses the same data,
but also improves its accuracy.

The paper is organized as follows. We start section \ref{Sec1} by briefly
recalling the definition of the generalized Taylor polynomial. Then we deal
with the extension $P_{m}^{\left[  \Delta_{2}\left(  V_{1},V_{2},V_{3}\right)
;V_{1}\right]  }[f]$ of the generalized Taylor polynomial to a generic simplex
$\Delta_{2}\left(  V_{1},V_{2},V_{3}\right)  $ of vertices $V_{1},V_{2}%
,V_{3}\in D$. In particular we provide new results concerning: error of
approximation, limit behaviour and interpolation conditions of the given
extension. In section \ref{Sec:4} we use these results to define the bivariate
Shepard-Bernoulli operators and to study their remainder terms and rate of
convergence. In section \ref{Sec:5} we apply the bivariate Shepard-Bernoulli
operators to the scattered data interpolation problem. The numerical results
on some commonly used test functions for scattered data approximation
\cite{RenCli,RenBro} show that the bivariate interpolation scheme proposed
here is comparable well with the better known operators QSHEP2D \cite{Ren2}
and CSHEP2D \cite{Ren4}. Finally, in section \ref{Sec:6} we draw conclusions.

\section{Further remarks on the generalized Taylor polynomial.\label{Sec1}}

\subsection{The univariate expansion in Bernoulli polynomials.}

The generalized Taylor polynomial \cite{Cos}\ is an expansion in Bernoulli
polynomials, i.e. in polynomials defined recursively by means of the following
relations \cite{Cos,Jor}
\begin{equation}
\left\{
\begin{array}
[c]{ll}%
B_{0}\left(  x\right)  =1, & \\
B_{n}^{\prime}\left(  x\right)  =nB_{n-1}\left(  x\right)  , & n\geq1,\\
{\displaystyle\int\limits_{0}^{1}}B_{n}\left(  x\right)  dx=0, & n\geq1.
\end{array}
\right.  \label{berricrel}%
\end{equation}
For functions in the class $C^{m}\left(  \left[  a,b\right]  \right)
,\ a,b\in\mathbb{R},\ a<b$, this expansion is realized by the equation
\begin{equation}
f\left(  x\right)  =P_{m}[f,a,b](x)+R_{m}\left[  f,a,b\right]  (x),\quad
x\in\left[  a,b\right]  , \label{taygenappfor}%
\end{equation}
the polynomial approximant is defined by
\begin{equation}
P_{m}[f,a,b](x)=f\left(  a\right)  +{\displaystyle\sum\limits_{k=1}^{m}}%
\frac{S_{k}\left(  \dfrac{x-a}{h}\right)  }{k!}h^{k-1}\left(  f^{(k-1)}\left(
b\right)  -f^{\left(  k-1\right)  }\left(  a\right)  \right)
\ \label{poltaygen}%
\end{equation}
and the remainder term is
\begin{equation}
R_{m}\left[  f,a,b\right]  (x)=\frac{h^{m-1}}{m!}{\displaystyle\int
\limits_{a}^{b}}f^{\left(  m\right)  }\left(  t\right)  \left(  B_{m}\left(
\frac{b-t}{h}\right)  -B_{m}\left(  \frac{\left(  x-t\right)  -\left[
x-t\right]  }{h}\right)  \right)  dt, \label{remtaygen}%
\end{equation}
where we have set $S_{k}\left(  x\right)  =B_{k}\left(  x\right)  -B_{k}$,
$B_{k}=B_{k}\left(  0\right)  $ and we have denoted by $\left[  \cdot\right]
$ the integer part of the argument and $h=b-a$. The polynomial $P_{m}[f,a,b]$
can be extended in a natural way to the whole real line; in this case Peano's
kernel theorem \cite[p. 70]{Dav}\ provides an integral expression for the
remainder (\ref{remtaygen}) \cite{CaiDel}. The main properties of the
generalized Taylor polynomial have been extensively studied in \cite{CaiDel,Cos}.

\subsection{The bivariate extension.}

\label{Sec:2.2}In \cite{CosDel2} the univariate expansion (\ref{taygenappfor})
has been extended to a bivariate expansion for functions of class $C^{m}$ in
the standard simplex $\Delta_{2}=\left\{  \left(  x,y\right)  \in
\mathbb{R}^{2}:x\geq0,y\geq0,x+y\leq1\right\}  $ which interpolates at the
vertices of the simplex and it is exact in $\mathcal{P}_{\boldsymbol{x}}^{m}$.
As mentioned in \cite{CosDel2}, this expansion can be generalized to a generic
simplex of $\mathbb{R}^{2}$ by means of a linear isomorphism. In this paper we
require this general expansion and, in order to formalize it, let us denote by
$%
%TCIMACRO{\U{2124} }%
%BeginExpansion
\mathbb{Z}
%EndExpansion
_{+}^{2}$ the set of all pairs with non-negative integer components in the
euclidean space $%
%TCIMACRO{\U{211d} }%
%BeginExpansion
\mathbb{R}
%EndExpansion
^{2}$. For $\beta=\left(  \beta_{1},\beta_{2}\right)  \in%
%TCIMACRO{\U{2124} }%
%BeginExpansion
\mathbb{Z}
%EndExpansion
_{+}^{2}$, we use the notations $\left\vert \beta\right\vert =\beta_{1}%
+\beta_{2}$, $\beta!=\beta_{1}!\beta_{2}!$ and $\beta\leq\alpha$ if and only
if $\beta_{i}\leq\alpha_{i}$ for all $i=1,2$. Moreover, we assume that $D$ is
a compact convex domain and $V_{1},V_{2},V_{3}\in D$. To fix the ideas we set
$V_{1}=(x_{1,}y_{1}),~V_{2}=\left(  x_{2},y_{2}\right)  ,~V_{3}=\left(
x_{3},y_{3}\right)  $. We denote by $\Delta_{2}\left(  V_{1},V_{2}%
,V_{3}\right)  $ the simplex of vertices $V_{1},V_{2},V_{3}$, i.e. the convex
hull of the set $\left\{  V_{1},V_{2},V_{3}\right\}  $. The baricentric
coordinates $\left(  \lambda_{1}\left(  \boldsymbol{x}\right)  ,\lambda
_{2}\left(  \boldsymbol{x}\right)  ,\lambda_{3}\left(  \boldsymbol{x}\right)
\right)  $, of a generic point $\boldsymbol{x}=\left(  x,y\right)  $%
$\in\mathbb{R}^{2}$, relative to the simplex $\Delta_{2}\left(  V_{1}%
,V_{2},V_{3}\right)  $, are defined by
\begin{equation}%
\begin{array}
[c]{ccc}%
\lambda_{1}\left(  \boldsymbol{x}\right)  =\dfrac{A\left(  \boldsymbol{x}%
,V_{2},V_{3}\right)  }{A\left(  V_{1},V_{2},V_{3}\right)  }, & \lambda
_{2}\left(  \boldsymbol{x}\right)  =\dfrac{A\left(  V_{1},\boldsymbol{x}%
,V_{3}\right)  }{A\left(  V_{1},V_{2},V_{3}\right)  }, & \lambda_{3}\left(
\boldsymbol{x}\right)  =\dfrac{A\left(  V_{1},V_{2},\boldsymbol{x}\right)
}{A\left(  V_{1},V_{2},V_{3}\right)  }%
\end{array}
\label{coorbar}%
\end{equation}
where $A\left(  V_{1},V_{2},V_{3}\right)  $ is the signed area of the simplex
of vertices $V_{1},V_{2},V_{3}$%
\[
A\left(  V_{1},V_{2},V_{3}\right)  =\left\vert
\begin{array}
[c]{ccc}%
1 & 1 & 1\\
x_{1} & x_{2} & x_{3}\\
y_{1} & y_{2} & y_{3}%
\end{array}
\right\vert .
\]
If $f$ is a differentiable function, and $V_{i}$ and $V_{j}$ are two distinct
vertices of the simplex $\Delta_{2}\left(  V_{1},V_{2},V_{3}\right)  $, the
derivative of $f$ along the directed line segment from $V_{i}$ to $V_{j}$
(side of the simplex) at $\boldsymbol{x}$ is denoted by%
\begin{equation}%
\begin{array}
[c]{cc}%
D_{ij}f\left(  \boldsymbol{x}\right)  :=\left(  V_{i}-V_{j}\right)
\cdot\nabla f\left(  \boldsymbol{x}\right)  , & i,j=1,2,3,i\neq j,
\end{array}
\label{D_ij(f)}%
\end{equation}
where $\cdot$ is the dot product and $\nabla f\left(  \boldsymbol{x}\right)
=\left(  \dfrac{\partial f}{\partial x}\left(  \boldsymbol{x}\right)
,\dfrac{\partial f}{\partial y}\left(  \boldsymbol{x}\right)  \right)  $. The
composition of derivatives along the directed sides of the simplex
(\ref{D_ij(f)}) are denoted by
\begin{equation}%
\begin{array}
[c]{ccc}%
D_{1}^{\beta}=D_{21}^{\beta_{1}}D_{31}^{\beta_{2}}, & D_{2}^{\beta}%
=D_{12}^{\beta_{1}}D_{32}^{\beta_{2}}, & D_{3}^{\beta}=D_{13}^{\beta_{1}%
}D_{23}^{\beta_{2}}.
\end{array}
\label{D0D1D2}%
\end{equation}
The following theorem holds:

\begin{theorem}
\label{gensimexp}Let $f$ be a function of class $C^{m}\left(  D\right)  $.
Then for each $\boldsymbol{x}\in\Delta_{2}\left(  V_{1},V_{2},V_{3}\right)  $
we have%
\begin{equation}
f\left(  \boldsymbol{x}\right)  =P_{m}^{\left[  \Delta_{2}\left(  V_{1}%
,V_{2},V_{3}\right)  ;V_{1}\right]  }[f](\boldsymbol{x})+R_{m}^{\left[
\Delta_{2}\left(  V_{1},V_{2},V_{3}\right)  ;V_{1}\right]  }\left[  f\right]
\left(  \boldsymbol{x}\right)  \label{gspoltayexp}%
\end{equation}
where%
\begin{equation}%
\begin{array}
[t]{l}%
P_{m}^{\left[  \Delta_{2}\left(  V_{1},V_{2},V_{3}\right)  ;V_{1}\right]
}[f](\boldsymbol{x})=f\left(  V_{1}\right) \\
\quad\quad+%
%TCIMACRO{\dsum \limits_{j=1}^{m}}%
%BeginExpansion
{\displaystyle\sum\limits_{j=1}^{m}}
%EndExpansion
\left(  D_{1}^{\left(  0,j-1\right)  }f\left(  V_{3}\right)  -D_{1}^{\left(
0,j-1\right)  }f\left(  V_{1}\right)  \right)  \dfrac{S_{j}\left(  \lambda
_{2}+\lambda_{3}\right)  }{j!}\\
\quad\quad+%
%TCIMACRO{\dsum \limits_{i=1}^{m}}%
%BeginExpansion
{\displaystyle\sum\limits_{i=1}^{m}}
%EndExpansion%
%TCIMACRO{\dsum \limits_{j=1}^{m-i+1}}%
%BeginExpansion
{\displaystyle\sum\limits_{j=1}^{m-i+1}}
%EndExpansion
\left(  \left(  -1\right)  ^{i+j}\left(  D_{2}^{\left(  j-1,i-1\right)
}f\left(  V_{2}\right)  -D_{2}^{\left(  j-1,i-1\right)  }f\left(
V_{1}\right)  \right)  \right. \\
\quad\quad\left.  +\left(  -1\right)  ^{j}\left(  D_{3}^{\left(
j-1,i-1\right)  }f\left(  V_{3}\right)  -D_{3}^{\left(  j-1,i-1\right)
}f\left(  V_{1}\right)  \right)  \right) \\
\quad\quad\times\dfrac{\left(  \lambda_{2}+\lambda_{3}\right)  ^{i-1}%
S_{i}\left(  \frac{\lambda_{2}}{\lambda_{2}+\lambda_{3}}\right)  }{i!}%
\dfrac{S_{j}\left(  \lambda_{2}+\lambda_{3}\right)  }{j!},
\end{array}
\label{gspoltaygen}%
\end{equation}
and $R_{m}^{\left[  \Delta_{2}\left(  V_{1},V_{2},V_{3}\right)  ;V_{1}\right]
}\left[  f\right]  \left(  \boldsymbol{x}\right)  $ is the remainder term.
\end{theorem}

\begin{proof}
[Proof sketch]Rather than obtaining expansion (\ref{gspoltaygen}) by using a
linear transformation which maps points $\left(  0,0\right)  ,\left(
1,0\right)  ,\left(  0,1\right)  $ in points $\left(  x_{1},y_{1}\right)
,\left(  x_{2},y_{2}\right)  ,\left(  x_{3},y_{3}\right)  $ respectively, we
proceed by adapting, in this general case, the extension technique to the
simplex specially developed in \cite{CosDel2,CosDel3} for expansions in
Bernoulli and Lidstone polynomials (see also \cite{CosDelGuz} for extensions
of asymmetric expansions).\ We denote by $r_{1}~$the line through $V_{2}%
,V_{3}$, by $r_{2}$ the line through $V_{1},V_{3}$ and by $r_{3}$ the line
through $V_{1},V_{2}$. Let $V=\left(  x,y\right)  \in\Delta_{2}\left(
V_{1},V_{2},V_{3}\right)  $ be an interior point. We denote by $r_{1}^{V}$ the
line through $V$ parallel to the line $r_{1}$ and by $T_{2}\left(  x,y\right)
=\left(  x-\lambda_{2}h_{1},y-\lambda_{2}k_{1}\right)  ,T_{3}\left(
x,y\right)  =\left(  x+\lambda_{3}h_{1},y+\lambda_{3}k_{1}\right)  $ the
intersection points between $r_{1}^{V}$ and $r_{2},r_{3}$, respectively, where
$\left(  h_{1},k_{1}\right)  =V_{2}-V_{3}$. We assign point $V$ to the line
segment $S\left(  x,y\right)  =\left[  T_{2},T_{3}\right]  $ parameterized by%
\[
\left\{
\begin{array}
[c]{cc}%
\begin{array}
[c]{c}%
x\left(  \lambda\right)  =x-\lambda_{2}h_{1}+\lambda\left(  \lambda
_{2}+\lambda_{3}\right)  h_{1}\\
y\left(  \lambda\right)  =y-\lambda_{2}k_{1}+\lambda\left(  \lambda
_{2}+\lambda_{3}\right)  k_{1}%
\end{array}
, & \lambda\in\left[  0,1\right]  .
\end{array}
\right.
\]
The restriction of $f$ to $S$ is the univariate function $f\left(  x\left(
\lambda\right)  ,y\left(  \lambda\right)  \right)  $ in $\left[  0,1\right]  $
and we expand it by generalized Taylor expansion (\ref{taygenappfor}). In this
expansion we replace parameter $\lambda$ by the value $\frac{\lambda_{2}%
}{\lambda_{2}+\lambda_{3}}$ which corresponds to point $V$. This results in an
expansion in terms of the values at $T_{2},T_{3}$ of $f$ and its derivatives
in direction of the side $V_{2},V_{3}$. We reach the vertices of the simplex
$\Delta_{2}\left(  V_{1},V_{2},V_{3}\right)  $ by assigning points
$T_{2},T_{3}$ to the segments with end points $V_{1},V_{3}$ and $V_{1},V_{2}$
respectively and by repeated use of the expansion (\ref{taygenappfor}) on $f$
and its directional derivatives (see fig. \ref{simplesso}).%
%TCIMACRO{\FRAME{ftbpFU}{2.4907in}{2.5028in}{0pt}{\Qcb{Let $V=\left(
%x,y\right)  \in\Delta_{2}\left(  V_{1},V_{2},V_{3}\right)  $. By assigning
%this point to the segment with end-points $T_{2},T_{3}$ we expand $f$ along
%this segment by generalized Taylor expansion; the vertices of the simplex are
%then reached by a repeated use of expansion (\ref{taygenappfor}) with respect
%to the points $V_{1},V_{3}$ and $V_{1},V_{2}$ on $f$ and its directional
%derivatives. }}{\Qlb{simplesso}}{simplesso.eps}%
%{\special{ language "Scientific Word";  type "GRAPHIC";
%maintain-aspect-ratio TRUE;  display "USEDEF";  valid_file "F";
%width 2.4907in;  height 2.5028in;  depth 0pt;  original-width 3.122in;
%original-height 3.1384in;  cropleft "0";  croptop "1";  cropright "1";
%cropbottom "0";  filename 'simplesso.eps';file-properties "XNPEU";}}}%
%BeginExpansion
\begin{figure}
[ptb]
\begin{center}
\includegraphics[
height=2.5028in,
width=3.5in
]%
{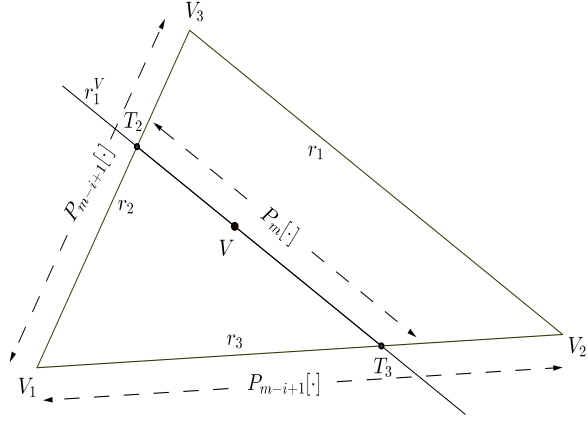}%
\caption{Let $V=\left(  x,y\right)  \in\Delta_{2}\left(  V_{1},V_{2}%
,V_{3}\right)  $. By assigning this point to the segment with end-points
$T_{2},T_{3}$ we expand $f$ along this segment by generalized Taylor
expansion; the vertices of the simplex are then reached by a repeated use of
expansion (\ref{taygenappfor}) with respect to the points $V_{1},V_{3}$ and
$V_{1},V_{2}$ on $f$ and its directional derivatives. }%
\label{simplesso}%
\end{center}
\end{figure}
%EndExpansion
\qquad
\end{proof}

If $\boldsymbol{x}\in\Delta_{2}\left(  V_{1},V_{2},V_{3}\right)  $ an
expression for the remainder $R_{m}^{\left[  \Delta_{2}\left(  V_{1}%
,V_{2},V_{3}\right)  ;V_{1}\right]  }\left[  f\right]  \left(  \boldsymbol{x}%
\right)  $ can be obtained by a repeated use of (\ref{remtaygen}) as in
\cite{CosDel2}. If $\boldsymbol{x}\in D$ the following Theorem\ provides an
expression for $R_{m}^{\left[  \Delta_{2}\left(  V_{1},V_{2},V_{3}\right)
;V_{1}\right]  }\left[  f\right]  \left(  \boldsymbol{x}\right)  $ under only
slightly stronger hypothesis on $f$. More precisely, we consider the class
$C^{m,1}\left(  D\right)  $ of functions $f\in C^{m}\left(  D\right)  $
\cite{Whit} with partial derivatives $\frac{\partial^{m}f}{\partial
x^{m-j}\partial y^{j}}$ Lipschitz-continuous in $D$ for each $j=0,\dots,m$ and
we set \cite{Far}%
\begin{equation}
\left\vert f\right\vert _{m,1}=\sup_{j=0,\dots,m}\left\{  \frac{\left\vert
\frac{\partial^{m}f}{\partial x^{m-j}\partial y^{j}}\left(  \boldsymbol{u}%
_{1}\right)  -\frac{\partial^{m}f}{\partial x^{m-j}\partial y^{j}}\left(
\boldsymbol{u}_{2}\right)  \right\vert }{\left\vert \boldsymbol{u}%
_{1}-\boldsymbol{u}_{2}\right\vert },\boldsymbol{u}_{1}\neq\boldsymbol{u}%
_{2}\text{ in }D\right\}  . \label{modfm,1}%
\end{equation}

\begin{theorem}
\label{gentaypol}Let $f$ be a function of class $C^{m,1}\left(  D\right)  $.
Then for each $\boldsymbol{x}\in D$ we have%
\begin{equation}
P_{m}^{\left[  \Delta_{2}\left(  V_{1},V_{2},V_{3}\right)  ;V_{1}\right]
}[f](\boldsymbol{x})=T_{m}\left[  f,V_{1}\right]  \left(  \boldsymbol{x}%
\right)  +\delta_{m}^{\left[  \Delta_{2}\left(  V_{1},V_{2},V_{3}\right)
;V_{1}\right]  }[f](\boldsymbol{x}) \label{PTplusDgen}%
\end{equation}
where $T_{m}\left[  f,V_{1}\right]  \left(  \boldsymbol{x}\right)  $ is the
Taylor polynomial of order $m$ for $f$ at $V_{1}$ \cite[Ch. 1]{Atk}\ and%
\begin{equation}%
\begin{array}
[t]{l}%
\delta_{m}^{\left[  \Delta_{2}\left(  V_{1},V_{2},V_{3}\right)  ;V_{1}\right]
}[f](\boldsymbol{x})=\\
\quad%
%TCIMACRO{\dsum \limits_{j=1}^{m}}%
%BeginExpansion
{\displaystyle\sum\limits_{j=1}^{m}}
%EndExpansion%
%TCIMACRO{\dint _{0}^{1}}%
%BeginExpansion
{\displaystyle\int_{0}^{1}}
%EndExpansion
\frac{D_{31}^{m-j+2}D_{1}^{\left(  0,j-1\right)  }f\left(  V_{1}+t\left(
V_{3}-V_{1}\right)  \right)  \left(  1-t\right)  ^{m-j+1}}{\left(
m-j+1\right)  !}dt\dfrac{S_{j}\left(  \lambda_{2}+\lambda_{3}\right)  }{j!}\\
\quad+%
%TCIMACRO{\dsum \limits_{i=1}^{m}}%
%BeginExpansion
{\displaystyle\sum\limits_{i=1}^{m}}
%EndExpansion%
%TCIMACRO{\dsum \limits_{j=1}^{m-i+1}}%
%BeginExpansion
{\displaystyle\sum\limits_{j=1}^{m-i+1}}
%EndExpansion
\left(  \left(  -1\right)  ^{i+j}%
%TCIMACRO{\dint _{0}^{1}}%
%BeginExpansion
{\displaystyle\int_{0}^{1}}
%EndExpansion
\frac{D_{21}^{m-j-i+3}D_{2}^{\left(  j-1,i-1\right)  }f\left(  V_{1}+t\left(
V_{2}-V_{1}\right)  \right)  \left(  1-t\right)  ^{m-j-i+2}}{(m-j-i+2)!}%
dt\right. \\
\quad+\left.  \left(  -1\right)  ^{j}%
%TCIMACRO{\dint _{0}^{1}}%
%BeginExpansion
{\displaystyle\int_{0}^{1}}
%EndExpansion
\frac{D_{31}^{m-j-i+3}D_{3}^{\left(  j-1,i-1\right)  }f\left(  V_{1}+t\left(
V_{3}-V_{1}\right)  \right)  \left(  1-t\right)  ^{m-j-i+2}}{(m-j-i+2)!}%
dt\right) \\
\quad\times\dfrac{\left(  \lambda_{2}+\lambda_{3}\right)  ^{i-1}S_{i}\left(
\frac{\lambda_{2}}{\lambda_{2}+\lambda_{3}}\right)  }{i!}\dfrac{S_{j}\left(
\lambda_{2}+\lambda_{3}\right)  }{j!}.
\end{array}
\label{Dgen}%
\end{equation}

\end{theorem}

\begin{proof}
By applying the Taylor theorem with integral remainder \cite[Ch. 7]{Apo} to
all derivative differences in (\ref{gspoltaygen}) we find%
\[
P_{m}^{\left[  \Delta_{2}\left(  V_{1},V_{2},V_{3}\right)  ;V_{1}\right]
}[f](\boldsymbol{x})=T_{m}\left[  f,V_{1}\right]  \left(  \boldsymbol{x}%
\right)  +\delta_{m}^{\left[  \Delta_{2}\left(  V_{1},V_{2},V_{3}\right)
;V_{1}\right]  }[f](\boldsymbol{x})
\]
where $\delta_{m}^{\left[  \Delta_{2}\left(  V_{1},V_{2},V_{3}\right)
;V_{1}\right]  }[f](\boldsymbol{x})$ is as in (\ref{Dgen}) and
\[%
\begin{array}
[c]{l}%
\widetilde{T}_{m}[f,V_{1}](\boldsymbol{x}):=f\left(  V_{1}\right)  +%
%TCIMACRO{\dsum \limits_{j=1}^{m}}%
%BeginExpansion
{\displaystyle\sum\limits_{j=1}^{m}}
%EndExpansion%
%TCIMACRO{\dsum \limits_{k=1}^{m-j+1}}%
%BeginExpansion
{\displaystyle\sum\limits_{k=1}^{m-j+1}}
%EndExpansion
\dfrac{1}{k!}D_{1}^{\left(  0,k+j-1\right)  }f\left(  V_{1}\right)
\dfrac{S_{j}\left(  \lambda_{2}+\lambda_{3}\right)  }{j!}\\
\quad+%
%TCIMACRO{\dsum \limits_{i=1}^{m}}%
%BeginExpansion
{\displaystyle\sum\limits_{i=1}^{m}}
%EndExpansion%
%TCIMACRO{\dsum \limits_{j=1}^{m-i+1}}%
%BeginExpansion
{\displaystyle\sum\limits_{j=1}^{m-i+1}}
%EndExpansion%
%TCIMACRO{\dsum \limits_{k=1}^{m-i-j+2}}%
%BeginExpansion
{\displaystyle\sum\limits_{k=1}^{m-i-j+2}}
%EndExpansion
\dfrac{1}{k!}\left(  \left(  -1\right)  ^{i+j+k}D_{2}^{\left(
j-1+k,i-1\right)  }f\left(  V_{1}\right)  +\left(  -1\right)  ^{j+k}%
D_{3}^{\left(  j-1+k,i-1\right)  }f\left(  V_{1}\right)  \right) \\
\quad\times\dfrac{\left(  \lambda_{2}+\lambda_{3}\right)  ^{i-1}S_{i}\left(
\frac{\lambda_{2}}{\lambda_{2}+\lambda_{3}}\right)  }{i!}\dfrac{S_{j}\left(
\lambda_{2}+\lambda_{3}\right)  }{j!}.
\end{array}
\]

Since $\delta_{m}^{\left[  \Delta_{2}\left(  V_{1},V_{2},V_{3}\right)
;V_{1}\right]  }[f](\boldsymbol{x})\equiv0$ for each $f\in\mathcal{P}%
_{\boldsymbol{x}}^{m}$ and $\dex\left(  P_{m}^{\left[  \Delta_{2}\left(
V_{1},V_{2},V_{3}\right)  ;V_{1}\right]  }\left[  \cdot\right]  \right)  =m$,
the polynomial operator $\widetilde{T}_{m}\left[  \cdot,V_{1}\right]  $
reproduces exactly polynomials up to the degree $m$. For this reason we can
affirm that $\widetilde{T}_{m}\left[  f,V_{1}\right]  $ is the $m$-th order
Taylor polynomial $T_{m}\left[  f,V_{1}\right]  $ for $f$ centered at $V_{1}$%
\begin{equation}
T_{m}\left[  f,V_{1}\right]  \left(  \boldsymbol{x}\right)  :=\sum
\limits_{\substack{\left\vert \alpha\right\vert \leq m\\\alpha\in%
%TCIMACRO{\U{2115} }%
%BeginExpansion
\mathbb{N}
%EndExpansion
_{0}^{2}}}\frac{1}{\alpha!}\frac{\partial^{\left\vert \alpha\right\vert }%
f}{\partial x^{\alpha_{1}}\partial y^{\alpha_{2}}}\left(  V_{1}\right)
\left(  x-x_{1}\right)  ^{\alpha_{1}}\left(  y-y_{1}\right)  ^{\alpha_{2}}.
\label{TaypolV0}%
\end{equation}

In fact, $\widetilde{T}_{m}\left[  f,V_{1}\right]  $ can be expressed, after
some computation and rearrangement, in terms of some partial derivatives of
$f$ up to the order $m$ at a point $V_{1}\in D$
\[
\widetilde{T}_{m}\left[  f,V_{1}\right]  \left(  \boldsymbol{x}\right)
=\sum\limits_{\substack{\left\vert \alpha\right\vert \leq m\\\alpha\in%
%TCIMACRO{\U{2115} }%
%BeginExpansion
\mathbb{N}
%EndExpansion
_{0}^{2}}}\frac{\partial^{\left\vert \alpha\right\vert }f}{\partial
x^{\alpha_{1}}\partial y^{\alpha_{2}}}\left(  V_{1}\right)  p_{\alpha}\left(
\boldsymbol{x}\right)  ,
\]
where $p_{\alpha}\left(  \boldsymbol{x}\right)  $ are polynomials of degree at
most $m$. Since $\widetilde{T}_{m}\left[  f,V_{1}\right]  =f$ for each
$f\in\mathcal{P}_{\boldsymbol{x}}^{m}$ then $\left\{  p_{\alpha}\left(
\boldsymbol{x}\right)  :\left\vert \alpha\right\vert \leq m\right\}  $
generates $\mathcal{P}_{\boldsymbol{x}}^{m}$, and therefore $p_{\alpha}\left(
\boldsymbol{x}\right)  \neq0$ for each $\alpha$. Since $\widetilde{T}%
_{m}\left[  \cdot,V_{1}\right]  $ reproduces exactly all polynomials in
$\mathcal{P}_{\boldsymbol{x}}^{m}$ it follows that%
\begin{equation}%
\begin{array}
[c]{cc}%
\widetilde{T}_{m}\left[  T_{m}\left[  f,V_{1}\right]  ,V_{1}\right]  \left(
\boldsymbol{x}\right)  =T_{m}\left[  f,V_{1}\right]  \left(  \boldsymbol{x}%
\right)  , & \boldsymbol{x}\in D;
\end{array}
\label{T_tilde_kequalT_k}%
\end{equation}
on the other hand%
\begin{equation}%
\begin{array}
[c]{cc}%
\widetilde{T}_{m}\left[  T_{m}\left[  f,V_{1}\right]  ,V_{1}\right]  \left(
\boldsymbol{x}\right)  =\widetilde{T}_{m}\left[  f,V_{1}\right]  \left(
\boldsymbol{x}\right)  , & \boldsymbol{x}\in D,
\end{array}
\label{T_tilde_k_f_V0equalT_tilde_kV0}%
\end{equation}
since for each $\alpha$, such that $\left\vert \alpha\right\vert \leq m$%
\[
\frac{\partial^{\left\vert \alpha\right\vert }T_{m}\left[  f,V_{1}\right]
}{\partial x^{\alpha_{1}}\partial y^{\alpha_{2}}}\left(  V_{1}\right)
=\frac{\partial^{\left\vert \alpha\right\vert }f}{\partial x^{\alpha_{1}%
}\partial y^{\alpha_{2}}}\left(  V_{1}\right)  .
\]
Therefore, by equaling the right hand side terms of (\ref{T_tilde_kequalT_k})
and (\ref{T_tilde_k_f_V0equalT_tilde_kV0}) we get%
\[
\widetilde{T}_{m}\left[  f,V_{1}\right]  \left(  \boldsymbol{x}\right)
=T_{m}\left[  f,V_{1}\right]  \left(  \boldsymbol{x}\right)  .
\]

\end{proof}

\begin{corollary}
\label{remgentaypol}In the hypothesis of Theorem \ref{gentaypol} equation
(\ref{gspoltayexp}) extends to all $\boldsymbol{x}\in D$ by setting%
\begin{equation}
R_{m}^{\left[  \Delta_{2}\left(  V_{1},V_{2},V_{3}\right)  ;V_{1}\right]
}\left[  f\right]  \left(  \boldsymbol{x}\right)  =R_{m}^{T}\left[
f,V_{1}\right]  \left(  x,y\right)  \left(  \boldsymbol{x}\right)  -\delta
_{m}^{\left[  \Delta_{2}\left(  V_{1},V_{2},V_{3}\right)  ;V_{1}\right]
}[f](\boldsymbol{x}) \label{RRTminusDgen}%
\end{equation}
where $R_{m}^{T}\left[  f,V_{1}\right]  \left(  \boldsymbol{x}\right)  $ is
the remainder in the Taylor expansion \cite[Ch. 1]{Atk}.
\end{corollary}

We are now able to give a bound for the remainder (\ref{RRTminusDgen}). From
here on, we shall use the notation $\left\vert \lambda\right\vert $ for the
absolute value of the real number $\lambda$ and $\left\Vert v\right\Vert _{2}$
for the modulus of the vector $v\in\mathbb{R}^{2}$. We set:%
\begin{equation}
r=\max\left\{  \left\Vert V_{1}-V_{2}\right\Vert _{2},\left\Vert V_{1}%
-V_{3}\right\Vert _{2},\left\Vert V_{2}-V_{3}\right\Vert _{2}\right\}  \text{
} \label{rdef}%
\end{equation}
and%
\begin{equation}
S^{-1}=A\left(  V_{1},V_{2},V_{3}\right)  . \label{Sdef}%
\end{equation}

By settings (\ref{coorbar}) and by the Cauchy-Schwarz inequality, we get the
following inequalities%
\begin{equation}%
\begin{array}
[t]{lll}%
\left\vert \lambda_{i}\right\vert  & \leq rS\left\Vert \boldsymbol{x}%
-V_{1}\right\Vert _{2}, & i=2,3,
\end{array}
\label{ineqmodl2l3}%
\end{equation}%
\begin{equation}%
\begin{array}
[t]{ll}%
\left\vert \lambda_{2}+\lambda_{3}\right\vert  & \leq rS\left\Vert
\boldsymbol{x}-V_{1}\right\Vert _{2}%
\end{array}
\label{ineqmodl2pl3}%
\end{equation}
and%
\begin{equation}%
\begin{array}
[c]{ccc}%
\left\vert \dfrac{\partial\lambda_{i}}{\partial x}\right\vert \leq rS, &
\left\vert \dfrac{\partial\lambda_{i}}{\partial y}\right\vert \leq rS, &
i=2,3.
\end{array}
\label{modDxlambdaiDylambai}%
\end{equation}
The following Lemma provides bounds for the derivatives of $f$ along the
directed sides of the simplex (\ref{D0D1D2}).

\begin{lemma}
\label{Lemma:modDjf}Let $f\in C^{m,1}\left(  D\right)  $. The derivatives of
$f$ along the directed sides of the simplex $\Delta_{2}\left(  V_{1}%
,V_{2},V_{3}\right)  $ satisfy
\begin{equation}%
\begin{array}
[c]{cccc}%
\left\vert D_{j}^{\beta}f\left(  \boldsymbol{x}\right)  \right\vert
\leq2^{\left\vert \beta\right\vert }r^{\left\vert \beta\right\vert }\left\vert
f\right\vert _{\left\vert \beta\right\vert -1,1}, & j=1,2,3 & \left\vert
\beta\right\vert =m+1, & \boldsymbol{x}\in D.
\end{array}
\label{seminorm2Di}%
\end{equation}

\end{lemma}

\begin{proof}
Let us consider the case $j=1$. By (\ref{D0D1D2}) and (\ref{D_ij(f)}) we have
\[%
\begin{array}
[c]{l}%
D_{1}^{\beta}f\left(  \boldsymbol{x}\right)  =D_{21}^{\beta_{1}}D_{31}%
^{\beta_{2}}f\left(  \boldsymbol{x}\right)  =\sum\limits_{i=0}^{\beta_{1}%
}\binom{\beta_{1}}{i}(x_{2}-x_{1})^{\beta_{1}-i}\left(  y_{2}-y_{1}\right)
^{i}\\
\quad\sum\limits_{j=0}^{\beta_{2}}\binom{\beta_{2}}{j}(x_{3}-x_{1})^{\beta
_{2}-j}\left(  y_{3}-y_{1}\right)  ^{j}\frac{\partial^{\left\vert
\beta\right\vert }f\left(  \boldsymbol{x}\right)  }{\partial x^{\left\vert
\beta\right\vert -i-j}\partial y^{i+j}}.
\end{array}
\]
Since $\left\vert \beta\right\vert \geq1$, at least one between $\left\vert
\beta\right\vert -i-j$ and $i+j$ is greater than or equal to $1$,
$i=0,\dots,\beta_{1}$, $j=0,\dots,\beta_{2}$. If $\left\vert \beta\right\vert
-i-j\geq1$%
\[
\left\vert \frac{\partial^{\left\vert \beta\right\vert }f\left(
\boldsymbol{x}\right)  }{\partial x^{\left\vert \beta\right\vert -i-j}\partial
y^{i+j}}\right\vert =\left\vert \lim_{h\rightarrow0}\frac{\frac{\partial
^{\left\vert \beta\right\vert -1}f\left(  x+h,y\right)  }{\partial
x^{\left\vert \beta\right\vert -i-j-1}\partial y^{i+j}}-\frac{\partial
^{\left\vert \beta\right\vert -1}f\left(  x,y\right)  }{\partial x^{\left\vert
\beta\right\vert -i-j-1}\partial y^{i+j}}}{h}\right\vert \leq\left\vert
f\right\vert _{\left\vert \beta\right\vert -1,1}%
\]
otherwise%
\[
\left\vert \frac{\partial^{\left\vert \beta\right\vert }f\left(
\boldsymbol{x}\right)  }{\partial x^{\left\vert \beta\right\vert -i-j}\partial
y^{i+j}}\right\vert =\left\vert \lim_{h\rightarrow0}\frac{\frac{\partial
^{\left\vert \beta\right\vert -1}f\left(  x,y+h\right)  }{\partial
x^{\left\vert \beta\right\vert -i-j}\partial y^{i+j-1}}-\frac{\partial
^{\left\vert \beta\right\vert -1}f\left(  x,y\right)  }{\partial x^{\left\vert
\beta\right\vert -i-j}\partial y^{i+j-1}}}{h}\right\vert \leq\left\vert
f\right\vert _{\left\vert \beta\right\vert -1,1}%
\]
$i=0,\dots,\beta_{1}$, $j=0,\dots,\beta_{2}$. Therefore we have%
\[%
\begin{array}
[c]{l}%
\left\vert D_{1}^{\beta}f\left(  \boldsymbol{x}\right)  \right\vert
=\left\vert \sum\limits_{i=0}^{\beta_{1}}\binom{\beta_{1}}{i}(x_{2}%
-x_{1})^{\beta_{1}-i}\left(  y_{2}-y_{1}\right)  ^{i}\right. \\
\qquad\left.  \sum\limits_{j=0}^{\beta_{2}}\binom{\beta_{2}}{j}(x_{3}%
-x_{1})^{\beta_{2}-j}\left(  y_{3}-y_{1}\right)  ^{j}\frac{\partial
^{\left\vert \beta\right\vert }f\left(  \boldsymbol{x}\right)  }{\partial
x^{\left\vert \beta\right\vert -i-j}\partial y^{i+j}}\right\vert \\
\qquad\leq\sum\limits_{i=0}^{\beta_{1}}\binom{\beta_{1}}{i}\left\vert
\left\vert x_{2}-x_{1}\right\vert \right\vert _{2}^{\beta_{1}-i}\left\vert
\left\vert y_{2}-y_{1}\right\vert \right\vert _{2}^{i}\\
\qquad\sum\limits_{j=0}^{\beta_{2}}\binom{\beta_{2}}{j}\left\vert \left\vert
x_{3}-x_{1}\right\vert \right\vert _{2}^{\beta_{2}-j}\left\vert \left\vert
y_{3}-y_{1}\right\vert \right\vert _{2}^{j}\left\vert \frac{\partial
^{\left\vert \beta\right\vert }f\left(  \boldsymbol{x}\right)  }{\partial
x^{\left\vert \beta\right\vert -i-j}\partial y^{i+j}}\right\vert \\
\qquad\leq2^{\left\vert \beta\right\vert }\left\vert \left\vert V_{2}%
-V_{1}\right\vert \right\vert _{2}^{\beta_{1}}\left\vert \left\vert
V_{3}-V_{1}\right\vert \right\vert _{2}^{\beta_{2}}\left\vert f\right\vert
_{\left\vert \beta\right\vert -1,1}\\
\qquad\leq2^{\left\vert \beta\right\vert }r^{\left\vert \beta\right\vert
}\left\vert f\right\vert _{\left\vert \beta\right\vert -1,1}.
\end{array}
\]
The cases $j=2,3$ are analogous.
\end{proof}

Now, we are able to prove the following Theorem.

\begin{theorem}
\label{deltabound}Let $f$ be a function of class $C^{m,1}\left(  D\right)  $.
Then for each $\boldsymbol{x}\in D$ we have the following bound for the
derivative $\dfrac{\partial^{\alpha+\beta}\delta_{m}^{\left[  \Delta
_{2}\left(  V_{1},V_{2},V_{3}\right)  ;V_{1}\right]  }\left[  f\right]
\left(  \boldsymbol{x}\right)  }{\partial x^{\alpha}\partial y^{\beta}}$ of
the difference (\ref{Dgen}), for each $\alpha,\beta:$ $0\leq\alpha+\beta\leq
m$%
\begin{equation}
\left\vert \dfrac{\partial^{\alpha+\beta}\delta_{m}^{\left[  \Delta_{2}\left(
V_{1},V_{2},V_{3}\right)  ;V_{1}\right]  }\left[  f\right]  \left(
\boldsymbol{x}\right)  }{\partial x^{\alpha}\partial y^{\beta}}\right\vert
\leq\left\vert f\right\vert _{m,1}C_{\alpha,\beta}\left(  m\right)
%TCIMACRO{\dsum \limits_{l=\max\left\{  1,\alpha+\beta\right\}  }^{m}}%
%BeginExpansion
{\displaystyle\sum\limits_{l=\max\left\{  1,\alpha+\beta\right\}  }^{m}}
%EndExpansion
r^{m+1-l}\left(  r^{2}S\right)  ^{l}\left\Vert \boldsymbol{x}-V_{1}\right\Vert
_{2}^{l-\left(  \alpha+\beta\right)  } \label{deltab}%
\end{equation}
where $C_{\alpha,\beta}\left(  m\right)  $ is a constant independent of $f$ or
$\boldsymbol{x}$ explicitly computable.
\end{theorem}

\begin{proof}
From properties of Bernoulli polynomials \cite{Jor}, \cite{CosDelGua}%
\[
B_{n}\left(  t\right)  =%
%TCIMACRO{\tsum \limits_{k=0}^{n}}%
%BeginExpansion
{\textstyle\sum\limits_{k=0}^{n}}
%EndExpansion
\dbinom{n}{k}B_{k}t^{n-k},\quad n=0,1,2,\ldots
\]
it follows that%
\begin{equation}%
\begin{array}
[t]{ll}%
S_{n}\left(  t\right)  & =B_{n}\left(  t\right)  -B_{n}\\
& =%
%TCIMACRO{\tsum \limits_{k=0}^{n-1}}%
%BeginExpansion
{\textstyle\sum\limits_{k=0}^{n-1}}
%EndExpansion
\dbinom{n}{k}B_{k}t^{n-k},\quad n=1,2,\ldots
\end{array}
\label{Snassum}%
\end{equation}
Moreover, we express the derivative $\dfrac{\partial^{\alpha+\beta}}{\partial
x^{\alpha}\partial y^{\beta}}$ as a combination of derivatives along the
directed line segments from $V_{2}$ to $V_{1}$ and from $V_{3}$ to $V_{1}$. In
fact, from relations (\ref{D_ij(f)}) we obtain%
\[%
\begin{array}
[c]{c}%
\frac{\partial}{\partial x}=\frac{\partial\lambda_{3}}{\partial x}D_{31}%
+\frac{\partial\lambda_{2}}{\partial x}D_{21},\\
\frac{\partial}{\partial y}=\frac{\partial\lambda_{3}}{\partial y}D_{31}%
+\frac{\partial\lambda_{2}}{\partial y}D_{21},
\end{array}
\]
and then%
\begin{equation}%
\begin{array}
[c]{c}%
\dfrac{\partial^{\alpha+\beta}\delta_{m}^{\left[  \Delta_{2}\left(
V_{1},V_{2},V_{3}\right)  ;V_{1}\right]  }\left[  f\right]  \left(
\boldsymbol{x}\right)  }{\partial x^{\alpha}\partial y^{\beta}}=%
%TCIMACRO{\dsum \limits_{p=0}^{\alpha}}%
%BeginExpansion
{\displaystyle\sum\limits_{p=0}^{\alpha}}
%EndExpansion%
%TCIMACRO{\dsum \limits_{q=0}^{\beta}}%
%BeginExpansion
{\displaystyle\sum\limits_{q=0}^{\beta}}
%EndExpansion
\dbinom{\alpha}{p}\dbinom{\beta}{q}\left(  \dfrac{\partial\lambda_{2}%
}{\partial x}\right)  ^{p}\left(  \dfrac{\partial\lambda_{3}}{\partial
x}\right)  ^{\alpha-p}\\
\times\left(  \dfrac{\partial\lambda_{2}}{\partial y}\right)  ^{q}\left(
\dfrac{\partial\lambda_{3}}{\partial y}\right)  ^{\beta-q}D_{21}^{p+q}%
D_{31}^{\alpha+\beta-p-q}\delta_{m}^{\left[  \Delta_{2}\left(  V_{1}%
,V_{2},V_{3}\right)  ;V_{1}\right]  }\left[  f\right]  \left(  \boldsymbol{x}%
\right)  .
\end{array}
\label{DxalphaDybetaf}%
\end{equation}
By taking the modulus of both sides of (\ref{DxalphaDybetaf}) and by using
relations (\ref{modDxlambdaiDylambai}) we get%
\begin{equation}
\left\vert \dfrac{\partial^{\alpha+\beta}\delta_{m}^{\left[  \Delta_{2}\left(
V_{1},V_{2},V_{3}\right)  ;V_{1}\right]  }[f](\boldsymbol{x})}{\partial
x^{\alpha}\partial y^{\beta}}\right\vert \leq%
%TCIMACRO{\dsum \limits_{p=0}^{\alpha}}%
%BeginExpansion
{\displaystyle\sum\limits_{p=0}^{\alpha}}
%EndExpansion%
%TCIMACRO{\dsum \limits_{q=0}^{\beta}}%
%BeginExpansion
{\displaystyle\sum\limits_{q=0}^{\beta}}
%EndExpansion
\dbinom{\alpha}{p}\dbinom{\beta}{q}\left(  rS\right)  ^{\alpha+\beta
}\left\vert D_{21}^{p+q}D_{31}^{\alpha+\beta-p-q}\delta_{m}^{\left[
\Delta_{2}\left(  V_{1},V_{2},V_{3}\right)  ;V_{1}\right]  }\left[  f\right]
\left(  \boldsymbol{x}\right)  \right\vert . \label{modDxalphaDybetaf}%
\end{equation}
Therefore we need to calculate and to bound
\begin{equation}
D_{21}^{\gamma_{1}}D_{31}^{\gamma_{2}}\delta_{m}^{\left[  \Delta_{2}\left(
V_{1},V_{2},V_{3}\right)  ;V_{1}\right]  }[f](\boldsymbol{x})
\label{D1gammaDelta}%
\end{equation}
for each $\gamma_{1}:=p+q$, $\gamma_{2}:=\alpha+\beta-p-q$, $p=0,\dots,\alpha
$; \thinspace$q=0,\dots,\beta$. In order to calculate (\ref{D1gammaDelta}) we
substitute relations (\ref{Snassum}) in equation (\ref{Dgen}) and, by the
Binomial Theorem, we get%
\begin{equation}%
\begin{array}
[t]{l}%
\delta_{m}^{\left[  \Delta_{2}\left(  V_{1},V_{2},V_{3}\right)  ;V_{1}\right]
}[f](\boldsymbol{x})=\\
\quad%
%TCIMACRO{\dsum \limits_{j=1}^{m}}%
%BeginExpansion
{\displaystyle\sum\limits_{j=1}^{m}}
%EndExpansion%
%TCIMACRO{\dint _{0}^{1}}%
%BeginExpansion
{\displaystyle\int_{0}^{1}}
%EndExpansion
\frac{D_{31}^{m-j+2}D_{1}^{\left(  0,j-1\right)  }f\left(  V_{1}+t\left(
V_{3}-V_{1}\right)  \right)  \left(  1-t\right)  ^{m-j+1}}{\left(
m-j+1\right)  !}dt\dfrac{%
%TCIMACRO{\tsum \limits_{k=0}^{j-1}}%
%BeginExpansion
{\textstyle\sum\limits_{k=0}^{j-1}}
%EndExpansion
\binom{j}{k}B_{k}\left(  \lambda_{2}+\lambda_{3}\right)  ^{j-k}}{j!}\\
\quad+%
%TCIMACRO{\dsum \limits_{i=1}^{m}}%
%BeginExpansion
{\displaystyle\sum\limits_{i=1}^{m}}
%EndExpansion%
%TCIMACRO{\dsum \limits_{j=1}^{m-i+1}}%
%BeginExpansion
{\displaystyle\sum\limits_{j=1}^{m-i+1}}
%EndExpansion
\left(  \left(  -1\right)  ^{i+j}%
%TCIMACRO{\dint _{0}^{1}}%
%BeginExpansion
{\displaystyle\int_{0}^{1}}
%EndExpansion
\frac{D_{21}^{m-j-i+3}D_{2}^{\left(  j-1,i-1\right)  }f\left(  V_{1}+t\left(
V_{2}-V_{1}\right)  \right)  \left(  1-t\right)  ^{m-j-i+2}}{(m-j-i+2)!}%
dt\right. \\
\quad+\left.  \left(  -1\right)  ^{j}%
%TCIMACRO{\dint _{0}^{1}}%
%BeginExpansion
{\displaystyle\int_{0}^{1}}
%EndExpansion
\frac{D_{31}^{m-j-i+3}D_{3}^{\left(  j-1,i-1\right)  }f\left(  V_{1}+t\left(
V_{3}-V_{1}\right)  \right)  \left(  1-t\right)  ^{m-j-i+2}}{(m-j-i+2)!}%
dt\right) \\
\quad\times\dfrac{%
%TCIMACRO{\tsum \limits_{k=0}^{i-1}}%
%BeginExpansion
{\textstyle\sum\limits_{k=0}^{i-1}}
%EndExpansion
\binom{i}{k}B_{k}}{i!}\dfrac{\sum\limits_{l=0}^{j-1}\binom{j}{l}B_{l}%
\sum\limits_{v=0}^{k+j-1-l}\binom{k+j-1-l}{v}\lambda_{2}^{i+j-1-l-v}%
\lambda_{3}^{v}}{j!}.
\end{array}
\label{deltamsem}%
\end{equation}
The application of the operator $D_{21}^{\gamma_{1}}D_{31}^{\gamma_{2}}$ to
$\delta_{m}^{\left[  \Delta_{2}\left(  V_{1},V_{2},V_{3}\right)
;V_{1}\right]  }[f](\boldsymbol{x})$ causes the disappearing of some addenda
on the right hand side of (\ref{deltamsem}) therefore, to highlight this
operation we make some changes of dummy index. Firstly we set $j-k=\kappa$ in
the first sequence of sums $%
%TCIMACRO{\dsum \limits_{j=1}^{m}}%
%BeginExpansion
{\displaystyle\sum\limits_{j=1}^{m}}
%EndExpansion%
%TCIMACRO{\dsum \limits_{k=1}^{j}}%
%BeginExpansion
{\displaystyle\sum\limits_{k=1}^{j}}
%EndExpansion
$\ and $i+j-1=\iota$ in the second sequence of sums $%
%TCIMACRO{\dsum \limits_{i=1}^{m}}%
%BeginExpansion
{\displaystyle\sum\limits_{i=1}^{m}}
%EndExpansion%
%TCIMACRO{\dsum \limits_{j=1}^{m-i+1}}%
%BeginExpansion
{\displaystyle\sum\limits_{j=1}^{m-i+1}}
%EndExpansion
$ and we get, by writing $k$ instead of $\kappa$ and $j$ instead of $\iota$
\[%
\begin{array}
[t]{l}%
\delta_{m}^{\left[  \Delta_{2}\left(  V_{1},V_{2},V_{3}\right)  ;V_{1}\right]
}[f](\boldsymbol{x})=\\
\quad%
%TCIMACRO{\dsum \limits_{j=1}^{m}}%
%BeginExpansion
{\displaystyle\sum\limits_{j=1}^{m}}
%EndExpansion%
%TCIMACRO{\dint _{0}^{1}}%
%BeginExpansion
{\displaystyle\int_{0}^{1}}
%EndExpansion
\frac{D_{31}^{m-j+2}D_{1}^{\left(  0,j-1\right)  }f\left(  V_{1}+t\left(
V_{3}-V_{1}\right)  \right)  \left(  1-t\right)  ^{m-j+1}}{\left(
m-j+1\right)  !}dt\dfrac{%
%TCIMACRO{\tsum \limits_{k=1}^{j}}%
%BeginExpansion
{\textstyle\sum\limits_{k=1}^{j}}
%EndExpansion
\binom{j}{j-\kappa}B_{j-k}\left(  \lambda_{2}+\lambda_{3}\right)  ^{k}}{j!}\\
\quad%
%TCIMACRO{\dsum \limits_{i=1}^{m}}%
%BeginExpansion
{\displaystyle\sum\limits_{i=1}^{m}}
%EndExpansion%
%TCIMACRO{\dsum \limits_{j=i}^{m}}%
%BeginExpansion
{\displaystyle\sum\limits_{j=i}^{m}}
%EndExpansion
\left(  \left(  -1\right)  ^{j+1}%
%TCIMACRO{\dint _{0}^{1}}%
%BeginExpansion
{\displaystyle\int_{0}^{1}}
%EndExpansion
\frac{D_{21}^{m-j-1+3}D_{2}^{\left(  j-i,i-1\right)  }f\left(  V_{1}+t\left(
V_{2}-V_{1}\right)  \right)  \left(  1-t\right)  ^{m-j+1}}{(m-j+1)!}dt\right.
\\
\quad+\left.  \left(  -1\right)  ^{j-i+1}%
%TCIMACRO{\dint _{0}^{1}}%
%BeginExpansion
{\displaystyle\int_{0}^{1}}
%EndExpansion
\frac{D_{31}^{m-j+2}D_{3}^{\left(  j-i,i-1\right)  }f\left(  V_{1}+t\left(
V_{3}-V_{1}\right)  \right)  \left(  1-t\right)  ^{m-j+1}}{(m-j+1)!}dt\right)
\\
\quad\times\dfrac{%
%TCIMACRO{\dsum \limits_{k=0}^{i-1}}%
%BeginExpansion
{\displaystyle\sum\limits_{k=0}^{i-1}}
%EndExpansion%
%TCIMACRO{\dsum \limits_{l=0}^{j-i}}%
%BeginExpansion
{\displaystyle\sum\limits_{l=0}^{j-i}}
%EndExpansion
\binom{j-i+1}{l}B_{l}%
%TCIMACRO{\dsum \limits_{v=0}^{k+j-i-l}}%
%BeginExpansion
{\displaystyle\sum\limits_{v=0}^{k+j-i-l}}
%EndExpansion
\binom{k+j-i-l}{v}\lambda_{2}^{j-l-v}\lambda_{3}^{v}}{i!(j-i+1)!}.
\end{array}
\]
Secondly we set $j-l=\lambda$ in the sequence of sums $%
%TCIMACRO{\dsum \limits_{j=i}^{m}}%
%BeginExpansion
{\displaystyle\sum\limits_{j=i}^{m}}
%EndExpansion%
%TCIMACRO{\dsum \limits_{l=i}^{j}}%
%BeginExpansion
{\displaystyle\sum\limits_{l=i}^{j}}
%EndExpansion
$ and we get, by writing $l$ instead of $\lambda$,%
\[%
\begin{array}
[t]{l}%
\delta_{m}^{\left[  \Delta_{2}\left(  V_{1},V_{2},V_{3}\right)  ;V_{1}\right]
}[f](\boldsymbol{x})=\\
\quad%
%TCIMACRO{\dsum \limits_{j=1}^{m}}%
%BeginExpansion
{\displaystyle\sum\limits_{j=1}^{m}}
%EndExpansion%
%TCIMACRO{\dint _{0}^{1}}%
%BeginExpansion
{\displaystyle\int_{0}^{1}}
%EndExpansion
\frac{D_{31}^{m-j+2}D_{1}^{\left(  0,j-1\right)  }f\left(  V_{1}+t\left(
V_{3}-V_{1}\right)  \right)  \left(  1-t\right)  ^{m-j+1}}{\left(
m-j+1\right)  !}dt\dfrac{%
%TCIMACRO{\tsum \limits_{k=1}^{j}}%
%BeginExpansion
{\textstyle\sum\limits_{k=1}^{j}}
%EndExpansion
\binom{j}{j-k}B_{j-k}\left(  \lambda_{2}+\lambda_{3}\right)  ^{k}}{j!}\\
\quad+%
%TCIMACRO{\dsum \limits_{i=1}^{m}}%
%BeginExpansion
{\displaystyle\sum\limits_{i=1}^{m}}
%EndExpansion%
%TCIMACRO{\dsum \limits_{j=i}^{m}}%
%BeginExpansion
{\displaystyle\sum\limits_{j=i}^{m}}
%EndExpansion
\left(  \left(  -1\right)  ^{j+1}%
%TCIMACRO{\dint _{0}^{1}}%
%BeginExpansion
{\displaystyle\int_{0}^{1}}
%EndExpansion
\frac{D_{21}^{m-j-1+3}D_{2}^{\left(  j-i,i-1\right)  }f\left(  V_{1}+t\left(
V_{2}-V_{1}\right)  \right)  \left(  1-t\right)  ^{m-j+1}}{(m-j+1)!}dt\right.
\\
\quad+\left.  \left(  -1\right)  ^{j-i+1}%
%TCIMACRO{\dint _{0}^{1}}%
%BeginExpansion
{\displaystyle\int_{0}^{1}}
%EndExpansion
\frac{D_{31}^{m-j+2}D_{3}^{\left(  j-i,i-1\right)  }f\left(  V_{1}+t\left(
V_{3}-V_{1}\right)  \right)  \left(  1-t\right)  ^{m-j+1}}{(m-j+1)!}dt\right)
\\
\quad\times\dfrac{\sum\limits_{k=0}^{i-1}\binom{i}{k}B_{k}\sum\limits_{l=i}%
^{j}\binom{j-i+1}{j-l}B_{j-l}\sum\limits_{v=0}^{k-i-l}\binom{k-i-l}{v}%
\lambda_{2}^{l-v}\lambda_{3}^{v}}{i!(j-i+1)!}.
\end{array}
\]
Thirdly we change the order of summations $%
%TCIMACRO{\dsum \limits_{j=1}^{m}}%
%BeginExpansion
{\displaystyle\sum\limits_{j=1}^{m}}
%EndExpansion%
%TCIMACRO{\dsum \limits_{k=1}^{j}}%
%BeginExpansion
{\displaystyle\sum\limits_{k=1}^{j}}
%EndExpansion
=%
%TCIMACRO{\dsum \limits_{k=1}^{m}}%
%BeginExpansion
{\displaystyle\sum\limits_{k=1}^{m}}
%EndExpansion%
%TCIMACRO{\dsum \limits_{j=k}^{m}}%
%BeginExpansion
{\displaystyle\sum\limits_{j=k}^{m}}
%EndExpansion
$ and $%
%TCIMACRO{\dsum \limits_{i=1}^{m}}%
%BeginExpansion
{\displaystyle\sum\limits_{i=1}^{m}}
%EndExpansion%
%TCIMACRO{\dsum \limits_{j=i}^{m}}%
%BeginExpansion
{\displaystyle\sum\limits_{j=i}^{m}}
%EndExpansion%
%TCIMACRO{\dsum \limits_{l=i}^{j}}%
%BeginExpansion
{\displaystyle\sum\limits_{l=i}^{j}}
%EndExpansion
=%
%TCIMACRO{\dsum \limits_{l=1}^{m}}%
%BeginExpansion
{\displaystyle\sum\limits_{l=1}^{m}}
%EndExpansion%
%TCIMACRO{\dsum \limits_{i=1}^{l}}%
%BeginExpansion
{\displaystyle\sum\limits_{i=1}^{l}}
%EndExpansion%
%TCIMACRO{\dsum \limits_{j=l}^{m}}%
%BeginExpansion
{\displaystyle\sum\limits_{j=l}^{m}}
%EndExpansion
$ and we get%
\[%
\begin{array}
[t]{l}%
\delta_{m}^{\left[  \Delta_{2}\left(  V_{1},V_{2},V_{3}\right)  ;V_{1}\right]
}[f](\boldsymbol{x})=\\
\quad%
%TCIMACRO{\dsum \limits_{k=1}^{m}}%
%BeginExpansion
{\displaystyle\sum\limits_{k=1}^{m}}
%EndExpansion%
%TCIMACRO{\dsum \limits_{j=k}^{m}}%
%BeginExpansion
{\displaystyle\sum\limits_{j=k}^{m}}
%EndExpansion%
%TCIMACRO{\dint _{0}^{1}}%
%BeginExpansion
{\displaystyle\int_{0}^{1}}
%EndExpansion
\frac{D_{31}^{m-j+2}D_{1}^{\left(  0,j-1\right)  }f\left(  V_{1}+t\left(
V_{3}-V_{1}\right)  \right)  \left(  1-t\right)  ^{m-j+1}}{\left(
m-j+1\right)  !}dt\dfrac{\tbinom{j}{j-k}B_{j-k}\left(  \lambda_{2}+\lambda
_{3}\right)  ^{k}}{j!}\\
\quad+%
%TCIMACRO{\dsum \limits_{l=1}^{m}}%
%BeginExpansion
{\displaystyle\sum\limits_{l=1}^{m}}
%EndExpansion%
%TCIMACRO{\dsum \limits_{i=1}^{l}}%
%BeginExpansion
{\displaystyle\sum\limits_{i=1}^{l}}
%EndExpansion%
%TCIMACRO{\dsum \limits_{j=l}^{m}}%
%BeginExpansion
{\displaystyle\sum\limits_{j=l}^{m}}
%EndExpansion
\left(  \left(  -1\right)  ^{j+1}%
%TCIMACRO{\dint _{0}^{1}}%
%BeginExpansion
{\displaystyle\int_{0}^{1}}
%EndExpansion
\frac{D_{21}^{m-j-1+3}D_{2}^{\left(  j-i,i-1\right)  }f\left(  V_{1}+t\left(
V_{2}-V_{1}\right)  \right)  \left(  1-t\right)  ^{m-j+1}}{(m-j+1)!}dt\right.
\\
\quad+\left.  \left(  -1\right)  ^{j-i+1}%
%TCIMACRO{\dint _{0}^{1}}%
%BeginExpansion
{\displaystyle\int_{0}^{1}}
%EndExpansion
\frac{D_{31}^{m-j+2}D_{3}^{\left(  j-i,i-1\right)  }f\left(  V_{1}+t\left(
V_{3}-V_{1}\right)  \right)  \left(  1-t\right)  ^{m-j+1}}{(m-j+1)!}dt\right)
\\
\quad\times\dfrac{\sum\limits_{k=0}^{i-1}\binom{i}{k}B_{k}\binom{j-i+1}%
{j-l}B_{j-l}\sum\limits_{v=0}^{k-i-l}\binom{k-i-l}{v}\lambda_{2}^{l-v}%
\lambda_{3}^{v}}{i!(j-i+1)!}.
\end{array}
\]
Now it is easy to calculate (\ref{D1gammaDelta}) by using the relations%
\[%
\begin{array}
[c]{cc}%
D_{21}\lambda_{2}=1, & D_{21}\lambda_{3}=0\\
D_{31}\lambda_{2}=0, & D_{31}\lambda_{3}=1.
\end{array}
\]
In fact%
\[
D_{21}^{\gamma_{1}}D_{31}^{\gamma_{2}}\left(  \lambda_{2}+\lambda_{3}\right)
^{k}=\left\{
\begin{array}
[c]{cc}%
\dfrac{k!}{\left(  k-\left\vert \gamma\right\vert \right)  !}\left(
\lambda_{2}+\lambda_{3}\right)  ^{k-\left\vert \gamma\right\vert }, &
k\geq\left\vert \gamma\right\vert ,\\
0, & \text{otherwise,}%
\end{array}
\right.
\]%
\[
D_{21}^{\gamma_{1}}D_{31}^{\gamma_{2}}\left(  \lambda_{2}^{l-v}\lambda_{3}%
^{v}\right)  =\left\{
\begin{array}
[c]{cc}%
\dfrac{\left(  l-v\right)  !}{\left(  l-v-\gamma_{1}\right)  !}\dfrac
{v!}{\left(  v-\gamma_{2}\right)  !}\lambda_{2}^{l-v-\gamma_{1}}\lambda
_{3}^{v-\gamma_{2}}, & \gamma_{2}\leq v\leq l-\gamma_{1},\\
0, & \text{otherwise,}%
\end{array}
\right.
\]
and by the change of dummy index $v+\gamma_{1}=\nu$ we get, by writing $v$
instead of $\nu$ and $l$ instead of $k$ in the first sequence of sums
\begin{equation}%
\begin{array}
[t]{l}%
D_{21}^{\gamma_{1}}D_{31}^{\gamma_{2}}\delta_{m}^{\left[  \Delta_{2}\left(
V_{1},V_{2},V_{3}\right)  ;V_{1}\right]  }[f](\boldsymbol{x})=\\
\quad%
%TCIMACRO{\dsum \limits_{l=\max\{1,\left\vert \gamma\right\vert \}}^{m}}%
%BeginExpansion
{\displaystyle\sum\limits_{l=\max\{1,\left\vert \gamma\right\vert \}}^{m}}
%EndExpansion%
%TCIMACRO{\dsum \limits_{j=l}^{m}}%
%BeginExpansion
{\displaystyle\sum\limits_{j=l}^{m}}
%EndExpansion%
%TCIMACRO{\dint _{0}^{1}}%
%BeginExpansion
{\displaystyle\int_{0}^{1}}
%EndExpansion
\frac{D_{31}^{m-j+2}D_{1}^{\left(  0,j-1\right)  }f\left(  V_{1}+t\left(
V_{3}-V_{1}\right)  \right)  \left(  1-t\right)  ^{m-j+1}}{\left(
m-j+1\right)  !}dt\dfrac{\binom{j}{j-l}B_{j-l}\dfrac{l!}{\left(  l-\left\vert
\gamma\right\vert \right)  !}\left(  \lambda_{2}+\lambda_{3}\right)
^{l-\left\vert \gamma\right\vert }}{j!}\\
\quad%
%TCIMACRO{\dsum \limits_{l=\max\{1,\left\vert \gamma\right\vert \}}^{m}}%
%BeginExpansion
{\displaystyle\sum\limits_{l=\max\{1,\left\vert \gamma\right\vert \}}^{m}}
%EndExpansion%
%TCIMACRO{\dsum \limits_{i=1}^{l}}%
%BeginExpansion
{\displaystyle\sum\limits_{i=1}^{l}}
%EndExpansion%
%TCIMACRO{\dsum \limits_{j=l}^{m}}%
%BeginExpansion
{\displaystyle\sum\limits_{j=l}^{m}}
%EndExpansion
\left(  \left(  -1\right)  ^{j+1}%
%TCIMACRO{\dint _{0}^{1}}%
%BeginExpansion
{\displaystyle\int_{0}^{1}}
%EndExpansion
\frac{D_{21}^{m-j-1+3}D_{2}^{\left(  j-i,i-1\right)  }f\left(  V_{1}+t\left(
V_{2}-V_{1}\right)  \right)  \left(  1-t\right)  ^{m-j+1}}{(m-j+1)!}dt\right.
\\
\quad+\left.  \left(  -1\right)  ^{j-i+1}%
%TCIMACRO{\dint _{0}^{1}}%
%BeginExpansion
{\displaystyle\int_{0}^{1}}
%EndExpansion
\frac{D_{31}^{m-j+2}D_{3}^{\left(  j-i,i-1\right)  }f\left(  V_{1}+t\left(
V_{3}-V_{1}\right)  \right)  \left(  1-t\right)  ^{m-j+1}}{(m-j+1)!}dt\right)
\\
\quad\times\dfrac{\sum\limits_{k=0}^{i-1}\binom{i}{k}B_{k}\binom{j-i+1}%
{j-l}B_{j-l}\sum\limits_{v=\left\vert \gamma\right\vert }^{l}\binom
{k-i-l}{v-\gamma_{1}}}{i!(j-i+1)!}\\
\quad\times\dfrac{\left(  l-v+\gamma_{1}\right)  !}{\left(  l-v\right)
!}\dfrac{\left(  v-\gamma_{1}\right)  !}{\left(  v-\left\vert \gamma
\right\vert \right)  !}\lambda_{2}^{l-v}\lambda_{3}^{v-\left\vert
\gamma\right\vert }.
\end{array}
\label{D21D31fin}%
\end{equation}

Now, by taking the modulus of both sides of (\ref{D21D31fin}) and by using
relations (\ref{ineqmodl2l3}), (\ref{ineqmodl2pl3}) and (\ref{seminorm2Di}) we
get%
\begin{equation}%
\begin{array}
[t]{l}%
\left\vert D_{21}^{\gamma_{1}}D_{31}^{\gamma_{2}}\delta_{m}^{\left[
\Delta_{2}\left(  V_{1},V_{2},V_{3}\right)  ;V_{1}\right]  }[f](\boldsymbol{x}%
)\right\vert \leq2^{m+1}\left\vert f\right\vert _{m,1}\\%
%TCIMACRO{\dsum \limits_{l=\max\left\{  1,\left\vert \gamma\right\vert
%\right\}  }^{m}}%
%BeginExpansion
{\displaystyle\sum\limits_{l=\max\left\{  1,\left\vert \gamma\right\vert
\right\}  }^{m}}
%EndExpansion
\left(
%TCIMACRO{\dsum \limits_{j=l}^{m}}%
%BeginExpansion
{\displaystyle\sum\limits_{j=l}^{m}}
%EndExpansion
\dfrac{\dbinom{j}{j-l}\left\vert B_{j-l}\right\vert \dfrac{l!}{\left(
l-\left\vert \gamma\right\vert \right)  !}}{\left(  m-j+2\right)  !j!}\right.
\\
\quad\left.  +%
%TCIMACRO{\dsum \limits_{i=1}^{l}}%
%BeginExpansion
{\displaystyle\sum\limits_{i=1}^{l}}
%EndExpansion%
%TCIMACRO{\dsum \limits_{j=l}^{m}}%
%BeginExpansion
{\displaystyle\sum\limits_{j=l}^{m}}
%EndExpansion
\dfrac{%
%TCIMACRO{\dsum \limits_{k=0}^{i-1}}%
%BeginExpansion
{\displaystyle\sum\limits_{k=0}^{i-1}}
%EndExpansion
2\dbinom{i}{k}\left\vert B_{k}\right\vert \binom{j-i+1}{j-l}\left\vert
B_{j-l}\right\vert
%TCIMACRO{\dsum \limits_{v=\left\vert \gamma\right\vert }^{l}}%
%BeginExpansion
{\displaystyle\sum\limits_{v=\left\vert \gamma\right\vert }^{l}}
%EndExpansion
\binom{k-i-l}{v-\gamma_{1}}}{(m-j+2)!i!(j-i+1)!}\right.  \\
\quad\left.  \times\dfrac{\left(  l-v+\gamma_{1}\right)  !}{\left(
l-v\right)  !}\dfrac{\left(  v-\gamma_{1}\right)  !}{\left(  v-\left\vert
\gamma\right\vert \right)  !}\right)  r^{m+1}\left(  rS\right)  ^{l-\left\vert
\gamma\right\vert }\left\Vert \boldsymbol{x}-V_{1}\right\Vert _{2}%
^{l-\left\vert \gamma\right\vert }%
\end{array}
\label{deltaless}%
\end{equation}
Finally by using inequality (\ref{modDxalphaDybetaf}) we get%
\[%
\begin{array}
[c]{l}%
\left\vert \dfrac{\partial^{\alpha+\beta}\delta_{m}^{\left[  \Delta_{2}\left(
V_{1},V_{2},V_{3}\right)  ;V_{1}\right]  }\left[  f\right]  \left(
\boldsymbol{x}\right)  }{\partial x^{\alpha}\partial y^{\beta}}\right\vert
\leq2^{m+1}\left\vert f\right\vert _{m,1}%
%TCIMACRO{\dsum \limits_{p=0}^{\alpha}}%
%BeginExpansion
{\displaystyle\sum\limits_{p=0}^{\alpha}}
%EndExpansion%
%TCIMACRO{\dsum \limits_{q=0}^{\beta}}%
%BeginExpansion
{\displaystyle\sum\limits_{q=0}^{\beta}}
%EndExpansion
\dbinom{\alpha}{p}\dbinom{\beta}{q}\\%
%TCIMACRO{\dsum \limits_{l=\max\left\{  1,\alpha+\beta\right\}  }^{m}}%
%BeginExpansion
{\displaystyle\sum\limits_{l=\max\left\{  1,\alpha+\beta\right\}  }^{m}}
%EndExpansion
\left(
%TCIMACRO{\dsum \limits_{j=l}^{m}}%
%BeginExpansion
{\displaystyle\sum\limits_{j=l}^{m}}
%EndExpansion
\dfrac{\dbinom{j}{j-l}\left\vert B_{j-l}\right\vert \dfrac{l!}{\left(
l-\alpha-\beta\right)  !}}{\left(  m-j+2\right)  !j!}\right.  \\
\quad\left.  +%
%TCIMACRO{\dsum \limits_{i=1}^{l}}%
%BeginExpansion
{\displaystyle\sum\limits_{i=1}^{l}}
%EndExpansion%
%TCIMACRO{\dsum \limits_{j=l}^{m}}%
%BeginExpansion
{\displaystyle\sum\limits_{j=l}^{m}}
%EndExpansion
\dfrac{%
%TCIMACRO{\dsum \limits_{k=0}^{i-1}}%
%BeginExpansion
{\displaystyle\sum\limits_{k=0}^{i-1}}
%EndExpansion
2\dbinom{i}{k}\left\vert B_{k}\right\vert \dbinom{j-i+1}{j-l}\left\vert
B_{j-l}\right\vert
%TCIMACRO{\dsum \limits_{v=\alpha+\beta}^{l}}%
%BeginExpansion
{\displaystyle\sum\limits_{v=\alpha+\beta}^{l}}
%EndExpansion
\dbinom{k-i-l}{v-\alpha}}{(m-j+2)!i!(j-i+1)!}\right.  \\
\quad\left.  \times\dfrac{\left(  l-v+p+q\right)  !}{\left(  l-v\right)
!}\dfrac{\left(  v-p-q\right)  !}{\left(  v-\alpha-\beta\right)  !}\right)
r^{m+1-l}\left(  r^{2}S\right)  ^{l}\left\Vert \boldsymbol{x}-V_{1}\right\Vert
_{2}^{l-\left(  \alpha+\beta\right)  }%
\end{array}
\]
We get the bound (\ref{deltab}) by the equality $%
%TCIMACRO{\dsum \limits_{p=0}^{\alpha}}%
%BeginExpansion
{\displaystyle\sum\limits_{p=0}^{\alpha}}
%EndExpansion%
%TCIMACRO{\dsum \limits_{q=0}^{\beta}}%
%BeginExpansion
{\displaystyle\sum\limits_{q=0}^{\beta}}
%EndExpansion
\dbinom{\alpha}{p}\dbinom{\beta}{q}=2^{\alpha+\beta}$ and by setting%
\begin{equation}%
\begin{array}
[c]{l}%
C_{\alpha,\beta}\left(  m\right)  =2^{m+1+\alpha+\beta}\left(
%TCIMACRO{\dsum \limits_{j=l}^{m}}%
%BeginExpansion
{\displaystyle\sum\limits_{j=l}^{m}}
%EndExpansion
\dfrac{\dbinom{j}{j-l}\left\vert B_{j-l}\right\vert \dfrac{l!}{\left(
l-\alpha-\beta\right)  !}}{\left(  m-j+2\right)  !j!}\right.  \\
\quad\left.  +%
%TCIMACRO{\dsum \limits_{i=1}^{l}}%
%BeginExpansion
{\displaystyle\sum\limits_{i=1}^{l}}
%EndExpansion%
%TCIMACRO{\dsum \limits_{j=l}^{m}}%
%BeginExpansion
{\displaystyle\sum\limits_{j=l}^{m}}
%EndExpansion
\dfrac{%
%TCIMACRO{\dsum \limits_{k=0}^{i-1}}%
%BeginExpansion
{\displaystyle\sum\limits_{k=0}^{i-1}}
%EndExpansion
2\dbinom{i}{k}\left\vert B_{k}\right\vert \dbinom{j-i+1}{j-l}\left\vert
B_{j-l}\right\vert
%TCIMACRO{\dsum \limits_{v=\alpha+\beta}^{l}}%
%BeginExpansion
{\displaystyle\sum\limits_{v=\alpha+\beta}^{l}}
%EndExpansion
\dbinom{k-i-l}{v-\alpha}}{(m-j+2)!i!(j-i+1)!}\right.  \\
\quad\left.  \times\dfrac{\left(  l-v+p+q\right)  !}{\left(  l-v\right)
!}\dfrac{\left(  v-p-q\right)  !}{\left(  v-\alpha-\beta\right)  !}\right)  .
\end{array}
\label{C_alpha_beta}%
\end{equation}

\end{proof}

\begin{remark}
Let us observe that the bound in Theorem (\ref{deltabound}) implies that for
$r\rightarrow0$
\[
\lim_{h\rightarrow0}P_{m}^{\left[  \Delta_{2}\left(  V_{1,}V_{1}+h\left(
V_{2}-V_{1}\right)  ,V_{1}+h\left(  V_{3}-V_{1}\right)  \right)
;V_{1}\right]  }[f](\boldsymbol{x})=T_{m}\left[  f,V_{1}\right]  \left(
\boldsymbol{x}\right)  .
\]

\end{remark}

\begin{corollary}
\label{TheoBound}In the hypothesis of Theorem \ref{deltabound} we have for
each $\boldsymbol{x}\in D$ the following bound for the error
(\ref{RRTminusDgen}):%
\begin{equation}%
\begin{array}
[c]{l}%
\left\vert \dfrac{\partial^{\alpha+\beta}R_{m}^{\left[  \Delta_{2}\left(
V_{1},V_{2},V_{3}\right)  ;V_{1}\right]  }\left[  f\right]  \left(
\boldsymbol{x}\right)  }{\partial x^{\alpha}\partial y^{\beta}}\right\vert
\leq\left\vert f\right\vert _{m,1}\left(  \dfrac{2^{m-\left(  \alpha
+\beta\right)  }}{\left(  m-\left(  \alpha+\beta\right)  -1\right)
!}\left\Vert \boldsymbol{x}-V_{1}\right\Vert _{2}^{m+1-\left(  \alpha
+\beta\right)  }\right. \\
\left.  \qquad+C_{\alpha,\beta}\left(  m\right)  \sum\limits_{l=\max\left\{
1,\alpha+\beta\right\}  }^{m}r^{m+1-l}\left(  r^{2}S\right)  ^{l}\left\Vert
\boldsymbol{x}-V_{1}\right\Vert _{2}^{l-\left(  \alpha+\beta\right)  }\right)
\end{array}
\label{rembound}%
\end{equation}

\end{corollary}

\begin{proof}
The thesis follows from (\ref{RRTminusDgen}) by an application of the triangle
inequality, by (\ref{deltab}) bounding the derivatives of Taylor remainder as
in \cite{Far}.
\end{proof}

\begin{remark}
The previous result gives a positive answer to a conjecture made\ in
\cite{CosDel2}. Therefore, we can refer to polynomial $P_{m}^{\left[
\Delta_{2}\left(  V_{1},V_{2},V_{3}\right)  ;V_{i}\right]  }[f]$ as the
generalized Taylor polynomial on $\Delta_{2}\left(  V_{1},V_{2},V_{3}\right)
$ with respect to the vertex $V_{i},\ i=1,2,3$. This polynomial joins other
bivariate polynomials with similar properties (see for example \cite{BrenScot}).
\end{remark}

\begin{corollary}
\label{Theoshape}In the hypothesis of Theorem \ref{gentaypol} for each
$\alpha$,$\beta\geq0$, $1\leq\alpha+\beta\leq m$ we have%
\begin{equation}
\frac{\partial^{\alpha+\beta}P_{m}^{\left[  \Delta_{2}\left(  V_{1}%
,V_{2},V_{3}\right)  ;V_{1}\right]  }[f]}{\partial x^{\alpha}\partial
y^{\beta}}(V_{1})=\frac{\partial^{\alpha+\beta}f}{\partial x^{\alpha}\partial
y^{\beta}}\left(  V_{1}\right)  +O\left(  r^{m+1-\left(  \alpha+\beta\right)
}\right)  . \label{derOhmmimjp1}%
\end{equation}

\end{corollary}

\begin{proof}
By Theorem \ref{gentaypol} and by interpolation conditions satisfied by the
Taylor polynomial \cite{Bojanov} it follows that%
\[
\frac{\partial^{\alpha+\beta}P_{m}^{\left[  \Delta_{2}\left(  V_{1}%
,V_{2},V_{3}\right)  ;V_{1}\right]  }[f]}{\partial x^{\alpha}\partial
y^{\beta}}(V_{1})-\frac{\partial^{\alpha+\beta}f}{\partial x^{\alpha}\partial
y^{\beta}}\left(  V_{1}\right)  =\frac{\partial^{\alpha+\beta}\delta
_{m}^{\left[  \Delta_{2}\left(  V_{1},V_{2},V_{3}\right)  ;V_{1}\right]
}\left[  f\right]  \left(  V_{1}\right)  }{\partial x^{\alpha}\partial
y^{\beta}}.
\]
Now we use the bound (\ref{deltab}) in the particular case $\boldsymbol{x}%
=V_{1}$ to get the thesis, recalling that $r^{2}S$ depends only on the shape
of the triangle.
\end{proof}

\begin{remark}
By rearranging the terms in the sums on the right hand side of
(\ref{gspoltaygen}) we note that%
\[
P_{1}^{\left[  \Delta_{2}\left(  V_{1},V_{2},V_{3}\right)  ;V_{1}\right]
}[f](\boldsymbol{x})=f\left(  V_{1}\right)  \lambda_{1}+f\left(  V_{2}\right)
\lambda_{2}+f\left(  V_{3}\right)  \lambda_{3}%
\]
is the Lagrange interpolant at the nodes $V_{1},V_{2},V_{3}$ and therefore it
does not depend on the choice of the vertex $V_{1}$; the polynomial
\[%
\begin{array}
[c]{c}%
P_{2}^{\left[  \Delta_{2}\left(  V_{1},V_{2},V_{3}\right)  ;V_{3}\right]
}[f]\left(  \boldsymbol{x}\right)  =P_{1}^{\left[  \Delta_{2}\left(
V_{1},V_{2},V_{3}\right)  ;V_{1}\right]  }[f](\boldsymbol{x})\\
\quad+\frac{1}{2}\lambda_{1}\lambda_{2}\left(  D_{2}^{\left(  1,0\right)
}f\left(  V_{2}\right)  -D_{2}^{\left(  1,0\right)  }f\left(  V_{1}\right)
\right)  \\
\quad+\frac{1}{2}\lambda_{1}\lambda_{3}\left(  D_{1}^{\left(  0,1\right)
}f\left(  V_{1}\right)  -D_{1}^{\left(  0,1\right)  }f\left(  V_{3}\right)
\right)  \\
\quad+\frac{1}{2}\lambda_{2}\lambda_{3}\left(  D_{3}^{\left(  0,1\right)
}f\left(  V_{3}\right)  -D_{3}^{\left(  0,1\right)  }f\left(  V_{2}\right)
\right)
\end{array}
\]
satisfies the same property and therefore joins well known quadratic
triangular finite elements \cite{BrenScot}. For $m\geq3$,
\[
P_{m}^{\left[  \Delta_{2}\left(  V_{1},V_{2},V_{3}\right)  ;V_{1}\right]
}[f]\left(  \boldsymbol{x}\right)  =P_{2}^{\left[  \Delta_{2}\left(
V_{1},V_{2},V_{3}\right)  ;V_{3}\right]  }[f]\left(  \boldsymbol{x}\right)
+\text{terms of degree at least }3,
\]
depends on the choice of the referring vertex.
\end{remark}

\begin{remark}
The polynomial $P_{2}^{\left[  \Delta_{2}\left(  V_{1},V_{2},V_{3}\right)
;V_{1}\right]  }[f]$ can be used to improve the accuracy of approximation of
the triangular Shepard method \cite{Little}.
\end{remark}

\section{The bivariate Shepard-Bernoulli operator}

\label{Sec:4}Let $V_{i}=\left(  x_{i},y_{i}\right)  ,i=1,\ldots,N$ be fixed
points of $D$; we set $\mathcal{N}=\left\{  V_{i},i=1,\ldots,N\right\}  $. We
now associate to each point $V_{i}$ a simplex with a vertex in $V_{i}$ for
each $i=1,\ldots,N$. Taking into account the bound (\ref{deltab}) in Theorem
\ref{deltabound}, for each fixed radius of influence $R_{w_{i}}$ about node
$V_{i}$ \cite{Ren2} we associate to $V_{i}$ the simplex $\Delta_{2}(i)\subset
B\left(  V_{i},R_{w_{i}}\right)  $ which minimizes the quantity $r_{i}\left(
r_{i}^{2}S_{i}\right)  $ where, as above, $r_{i}$ is the length of the longest
side of the simplex $\Delta_{2}\left(  i\right)  $ and $S_{i}$ is twice the
area of $\Delta_{2}\left(  i\right)  $. If $\alpha_{i},\beta_{i}$ denote the
adjacent angles to the side of length $r_{i}$, then $r_{i}^{2}S_{i}=\frac
{\sin\left(  \alpha_{i}+\beta_{i}\right)  }{\sin\alpha_{i}\sin\beta_{i}}$
depends only on the form of the triangle $\Delta_{2}\left(  i\right)  $. Such
a procedure can be well-defined if the following steps are followed:

\begin{enumerate}
\item enumerate the $N_{i}$ nodes in the closed ball $B\left(  V_{i},R_{w_{i}%
}\right)  $ according to increasing distance from $V_{i}$ using the induced
order of the given set of interpolation nodes;

\item enumerate the triangles according increasing order of the vertices;

\item get the first useful triangle.
\end{enumerate}

\begin{definition}
\label{DefShepBer}For each fixed $\mu>0$ and $m=1,2,\ldots$ the bivariate
Shepard-Bernoulli operator is defined by
\begin{equation}%
\begin{array}
[c]{cc}%
S_{B_{m}}\left[  f\right]  \left(  \boldsymbol{x}\right)  =\sum\limits_{i=1}%
^{N}\widetilde{W}_{\mu,i}\left(  \boldsymbol{x}\right)  P_{m}^{\Delta_{2}%
(i)}[f](\boldsymbol{x}), & \boldsymbol{x}\in D
\end{array}
\label{Bshep}%
\end{equation}
where $P_{m}^{\Delta_{2}(i)}[f](\boldsymbol{x}),\ i=1,\ldots,N$ is the
generalized Taylor polynomial (\ref{gspoltaygen}) over $D$. The remainder
term is%
\begin{equation}%
\begin{array}
[t]{ll}%
R_{B_{m}}\left[  f\right]  \left(  \boldsymbol{x}\right)  =f\left(
\boldsymbol{x}\right)  -S_{B_{m}}\left[  f\right]  \left(  \boldsymbol{x}%
\right)  , & \boldsymbol{x}\in D.
\end{array}
\label{Rshep}%
\end{equation}

\end{definition}

The following statements can be checked without any difficulty.

\begin{theorem}
The operator $S_{B_{m}}\left[  \cdot\right]  $ is an interpolation operator in
$V_{i},$ $i=1,...,N$.
\end{theorem}

\begin{proof}
In fact $P_{m}^{\Delta_{2}(i)}[f](\boldsymbol{x})$ interpolates at $V_{i},$
$i=1,\ldots,N$ and the assertion follows in view of the fact that the Shepard
basis is cardinal:%
\begin{equation}
\widetilde{W}_{\mu,i}\left(  x_{k},y_{k}\right)  =\delta_{ik},\quad
i,k=1,...,N. \label{cardinality}%
\end{equation}

\end{proof}

\begin{theorem}
The degree of exactness of the operator $S_{B_{m}}\left[  \cdot\right]  $ is
$m$, i.e. $S_{B_{m}}\left[  p\right]  =p$ for each bivariate polynomial
$p\in\mathcal{P}_{\boldsymbol{x}}^{m}$.
\end{theorem}

\begin{proof}
The assertion follows from the fact that the Shepard basis is a partition of
unity:
\begin{equation}
\sum\limits_{i=1}^{N}\widetilde{W}_{\mu,i}\left(  \boldsymbol{x}\right)
\equiv1, \label{parunity}%
\end{equation}
since the degree of exactness of $P_{m}^{\Delta_{2}(i)}[f]$ is $m$ for
$i=1,...,N.$
\end{proof}

As for the continuity class of the Shepard operator, and consequently the
continuity class of the Shepard-Bernoulli operators, there is the following
result \cite{Bar}.

\begin{theorem}
If $P[\cdot,x_{i}],\ i=1,...,N$ are polynomial interpolation operators in
$x_{i}$, then the continuity class of the operator (\ref{Bshep}) depends upon
$\mu$ and, for $\mu>0$, is as follows:

\begin{itemize}
\item[i)] if $\mu$ is an integer, then $S_{N,\mu}P\left[  \cdot\right]  \in
C^{\mu-1}$;

\item[ii)] if $\mu$ is not an integer, then $S_{N,\mu}P\left[  \cdot\right]
\in C^{[\mu]}$;
\end{itemize}

here $[\mu]$ is the largest integer $<\mu$.
\end{theorem}

\begin{theorem}
For each $\alpha,\beta\in\mathbb{N}$ s.t. $1\leq\alpha+\beta<\mu$ we have%
\[%
\begin{array}
[c]{ll}%
\dfrac{\partial^{\alpha+\beta}}{\partial x^{\alpha}\partial y^{\beta}}%
S_{B_{m}}\left[  f\right]  \left(  V_{k}\right)   & =\dfrac{\partial
^{\alpha+\beta}}{\partial x^{\alpha}\partial y^{\beta}}P_{m}^{\Delta_{2}%
(k)}[f]\left(  V_{k}\right)  \\
& =\dfrac{\partial^{\alpha+\beta}f}{\partial x^{\alpha}\partial y^{\beta}%
}\left(  V_{k}\right)  +O\left(  r_{k}^{m+1-\left(  \alpha+\beta\right)
}\right)
\end{array}
\]
for each $k=1,\ldots,N$.
\end{theorem}

\begin{proof}
It follows from the known relation \cite{GorWix}
\begin{equation}
\dfrac{\partial^{\alpha+\beta}}{\partial x^{\alpha}\partial y^{\beta}%
}\widetilde{W}_{\mu,i}\left(  V_{k}\right)  =0 \label{derivative}%
\end{equation}
which holds for $\ i,k=1,...,N$, $1\leq\alpha+\beta$ $<\mu$ by applying the
Leibniz rule and by using relations (\ref{cardinality}), (\ref{parunity}) and
finally equation (\ref{derOhmmimjp1}).
\end{proof}

Convergence results can be obtained by following the known approaches
\cite[\S \ 15.4]{Wen}, \cite{Zup}. We set:

\begin{enumerate}
\item \label{1}$\mathcal{I}_{\boldsymbol{x}}=\left\{  i\in\left\{
1,\dots,N\right\}  :\left\vert \left\vert \boldsymbol{x}-V_{i}\right\vert
\right\vert _{2}<R_{w_{i}}\right\}  $, $\boldsymbol{x}\in D$;

\item $M=\sup\limits_{\boldsymbol{x}\in D}card\left\{  \mathcal{I}%
_{\boldsymbol{x}}\right\}  $;

\item $d_{i}=2R_{w_{i}}$, $i=1,\dots,N$,

\item \label{4}$K_{\alpha,\beta}>0$, $\alpha,\beta\in%
%TCIMACRO{\U{2115} }%
%BeginExpansion
\mathbb{N}
%EndExpansion
:$ $0\leq\alpha+\beta\leq\mu-1$ constants satisfying
\[
\sup_{\boldsymbol{x}\in B\left(  \boldsymbol{V}_{i},R_{w_{i}}\right)
}\left\vert \frac{\partial^{\alpha+\beta}\widetilde{W}_{\mu,i}\left(
\boldsymbol{x}\right)  }{\partial x^{\alpha}\partial y^{\beta}}\right\vert
\leq\frac{K_{\alpha,\beta}}{d_{i}^{\alpha+\beta}};
\]

\end{enumerate}

\begin{theorem}
\label{boundderf}Let $f$ be a function of class $C^{m,1}\left(  D\right)  $.
Then for each $\alpha,\beta\in%
%TCIMACRO{\U{2115} }%
%BeginExpansion
\mathbb{N}
%EndExpansion
:$ $0\leq\alpha+\beta<\mu$ the following bound holds%
\[%
\begin{array}
[c]{l}%
\sup\limits_{\boldsymbol{x}\in D}\left\vert \dfrac{\partial^{\alpha+\beta
}R_{B_{m}}\left[  f\right]  }{\partial x^{\alpha}\partial y^{\beta}%
}\right\vert \leq\left\vert f\right\vert _{m,1}M%
%TCIMACRO{\dsum \limits_{\substack{0\leq\gamma_{1}\leq\alpha\\0\leq\gamma
%_{2}\leq\beta}}}%
%BeginExpansion
{\displaystyle\sum\limits_{\substack{0\leq\gamma_{1}\leq\alpha\\0\leq
\gamma_{2}\leq\beta}}}
%EndExpansion
\dbinom{\alpha}{\gamma_{1}}\dbinom{\beta}{\gamma_{2}}\\
\quad\max\limits_{i\in\mathcal{I}_{\boldsymbol{x}}}\left\{  \dfrac
{K_{\alpha-\gamma_{1},\beta-\gamma_{2}}}{d_{i}^{\alpha+\beta-\gamma_{1}%
-\gamma_{2}}}\left(  \dfrac{2^{m-\left(  \gamma_{1}+\gamma_{2}\right)  }%
}{\left(  m-\left(  \gamma_{1}+\gamma_{2}\right)  -1\right)  !}R_{w_{i}%
}^{m+1-\left(  \gamma_{1}+\gamma_{2}\right)  }\right.  \right.  \\
\left.  \left.  \qquad+C_{\gamma_{1},\gamma_{2}}\left(  m\right)
\sum\limits_{l=\max\left\{  1,\gamma_{1}+\gamma_{2}\right\}  }^{m}%
r_{i}^{m+1-l}\left(  r_{i}^{2}S_{i}\right)  ^{l}R_{w_{i}}^{l-\left(
\gamma_{1}+\gamma_{2}\right)  }\right)  \right\}
\end{array}
\]
with $C_{\gamma_{1},\gamma_{2}}\left(  m\right)  $ defined in
(\ref{C_alpha_beta}).
\end{theorem}

\begin{proof}
By differentiating $\alpha$ times with respect to $x$ and $\beta$ times with
respect to $y$, $0\leq\alpha+\beta<\mu$, both sides of (\ref{Rshep}), by using
Leibniz' rule, we get%
\[%
\begin{array}
[c]{l}%
\dfrac{\partial^{\alpha+\beta}R_{B_{m}}\left[  f\right]  \left(
\boldsymbol{x}\right)  }{\partial x^{\alpha}\partial y^{\beta}}=\\%
%TCIMACRO{\dsum \limits_{i\in\mathcal{I}_{\boldsymbol{x}}}}%
%BeginExpansion
{\displaystyle\sum\limits_{i\in\mathcal{I}_{\boldsymbol{x}}}}
%EndExpansion%
%TCIMACRO{\dsum \limits_{\substack{0\leq\gamma_{1}\leq\alpha\\0\leq\gamma
%_{2}\leq\beta}}}%
%BeginExpansion
{\displaystyle\sum\limits_{\substack{0\leq\gamma_{1}\leq\alpha\\0\leq
\gamma_{2}\leq\beta}}}
%EndExpansion
\dbinom{\alpha}{\gamma_{1}}\dbinom{\beta}{\gamma_{2}}\dfrac{\partial
^{\alpha+\beta-\gamma_{1}-\gamma_{2}}\widetilde{W}_{\mu,i}\left(
\boldsymbol{x}\right)  }{\partial x^{\alpha-\gamma_{1}}\partial y^{\beta
-\gamma_{2}}}\dfrac{\partial^{\gamma_{1}+\gamma_{2}}R_{m}^{\Delta_{2}\left(
i\right)  }\left[  f\right]  \left(  \boldsymbol{x}\right)  }{\partial
x^{\gamma_{1}}\partial y^{\gamma_{2}}}.
\end{array}
\]
therefore
\[%
\begin{array}
[c]{l}%
\left\vert \dfrac{\partial^{\alpha+\beta}R_{B_{m}}\left[  f\right]  \left(
\boldsymbol{x}\right)  }{\partial x^{\alpha}\partial y^{\beta}}\right\vert
\leq\\
\qquad%
%TCIMACRO{\dsum \limits_{i\in\mathcal{I}_{\boldsymbol{x}}}}%
%BeginExpansion
{\displaystyle\sum\limits_{i\in\mathcal{I}_{\boldsymbol{x}}}}
%EndExpansion%
%TCIMACRO{\dsum \limits_{\substack{0\leq\gamma_{1}\leq\alpha\\0\leq\gamma
%_{2}\leq\beta}}}%
%BeginExpansion
{\displaystyle\sum\limits_{\substack{0\leq\gamma_{1}\leq\alpha\\0\leq
\gamma_{2}\leq\beta}}}
%EndExpansion
\dbinom{\alpha}{\gamma_{1}}\dbinom{\beta}{\gamma_{2}}\left\vert \dfrac
{\partial^{\alpha+\beta-\gamma_{1}-\gamma_{2}}\widetilde{W}_{\mu,i}\left(
\boldsymbol{x}\right)  }{\partial x^{\alpha-\gamma_{1}}\partial y^{\beta
-\gamma_{2}}}\right\vert \left\vert \dfrac{\partial^{\gamma_{1}+\gamma_{2}%
}R_{m}^{\Delta_{2}\left(  i\right)  }\left[  f\right]  \left(  \boldsymbol{x}%
\right)  }{\partial x^{\gamma_{1}}\partial y^{\gamma_{2}}}\right\vert \\
\quad\leq%
%TCIMACRO{\dsum \limits_{i\in\mathcal{I}_{\boldsymbol{x}}}}%
%BeginExpansion
{\displaystyle\sum\limits_{i\in\mathcal{I}_{\boldsymbol{x}}}}
%EndExpansion%
%TCIMACRO{\dsum \limits_{\substack{0\leq\gamma_{1}\leq\alpha\\0\leq\gamma
%_{2}\leq\beta}}}%
%BeginExpansion
{\displaystyle\sum\limits_{\substack{0\leq\gamma_{1}\leq\alpha\\0\leq
\gamma_{2}\leq\beta}}}
%EndExpansion
\dbinom{\alpha}{\gamma_{1}}\dbinom{\beta}{\gamma_{2}}\dfrac{K_{\alpha
-\gamma_{1},\beta-\gamma_{2}}}{d_{i}^{\alpha+\beta-\gamma_{1}-\gamma_{2}}%
}\left\vert f\right\vert _{m,1}\left(  \dfrac{2^{m-\left(  \gamma_{1}%
+\gamma_{2}\right)  }}{\left(  m-\left(  \gamma_{1}+\gamma_{2}\right)
-1\right)  !}\left\Vert \boldsymbol{x}-V_{i}\right\Vert _{2}^{m+1-\left(
\gamma_{1}+\gamma_{2}\right)  }\right.  \\
\left.  \qquad+C_{\gamma_{1},\gamma_{2}}\left(  m\right)  \sum\limits_{l=\max
\left\{  1,\gamma_{1}+\gamma_{2}\right\}  }^{m}r_{i}^{m+1-l}\left(  r_{i}%
^{2}S_{i}\right)  ^{l}\left\Vert \boldsymbol{x}-V_{i}\right\Vert
_{2}^{l-\left(  \gamma_{1}+\gamma_{2}\right)  }\right)  \\
\quad\leq\left\vert f\right\vert _{m,1}M%
%TCIMACRO{\dsum \limits_{\substack{0\leq\gamma_{1}\leq\alpha\\0\leq\gamma
%_{2}\leq\beta}}}%
%BeginExpansion
{\displaystyle\sum\limits_{\substack{0\leq\gamma_{1}\leq\alpha\\0\leq
\gamma_{2}\leq\beta}}}
%EndExpansion
\dbinom{\alpha}{\gamma_{1}}\dbinom{\beta}{\gamma_{2}}\max\limits_{i\in
\mathcal{I}_{\boldsymbol{x}}}\left\{  \dfrac{K_{\alpha-\gamma_{1},\beta
-\gamma_{2}}}{d_{i}^{\alpha+\beta-\gamma_{1}-\gamma_{2}}}\left(
\dfrac{2^{m-\left(  \gamma_{1}+\gamma_{2}\right)  }}{\left(  m-\left(
\gamma_{1}+\gamma_{2}\right)  -1\right)  !}R_{w_{i}}^{m+1-\left(  \gamma
_{1}+\gamma_{2}\right)  }\right.  \right.  \\
\left.  \left.  \qquad+C_{\gamma_{1},\gamma_{2}}\left(  m\right)
\sum\limits_{l=\max\left\{  1,\gamma_{1}+\gamma_{2}\right\}  }^{m}%
r_{i}^{m+1-l}\left(  r_{i}^{2}S_{i}\right)  ^{l}R_{w_{i}}^{l-\left(
\gamma_{1}+\gamma_{2}\right)  }\right)  \right\}  .
\end{array}
\]

\end{proof}

In the following section we present numerical results that testify\ to the
accuracy of the proposed operator.

\section{Numerical tests.}

\begin{figure}[ptb]
\includegraphics[width=.33\linewidth]{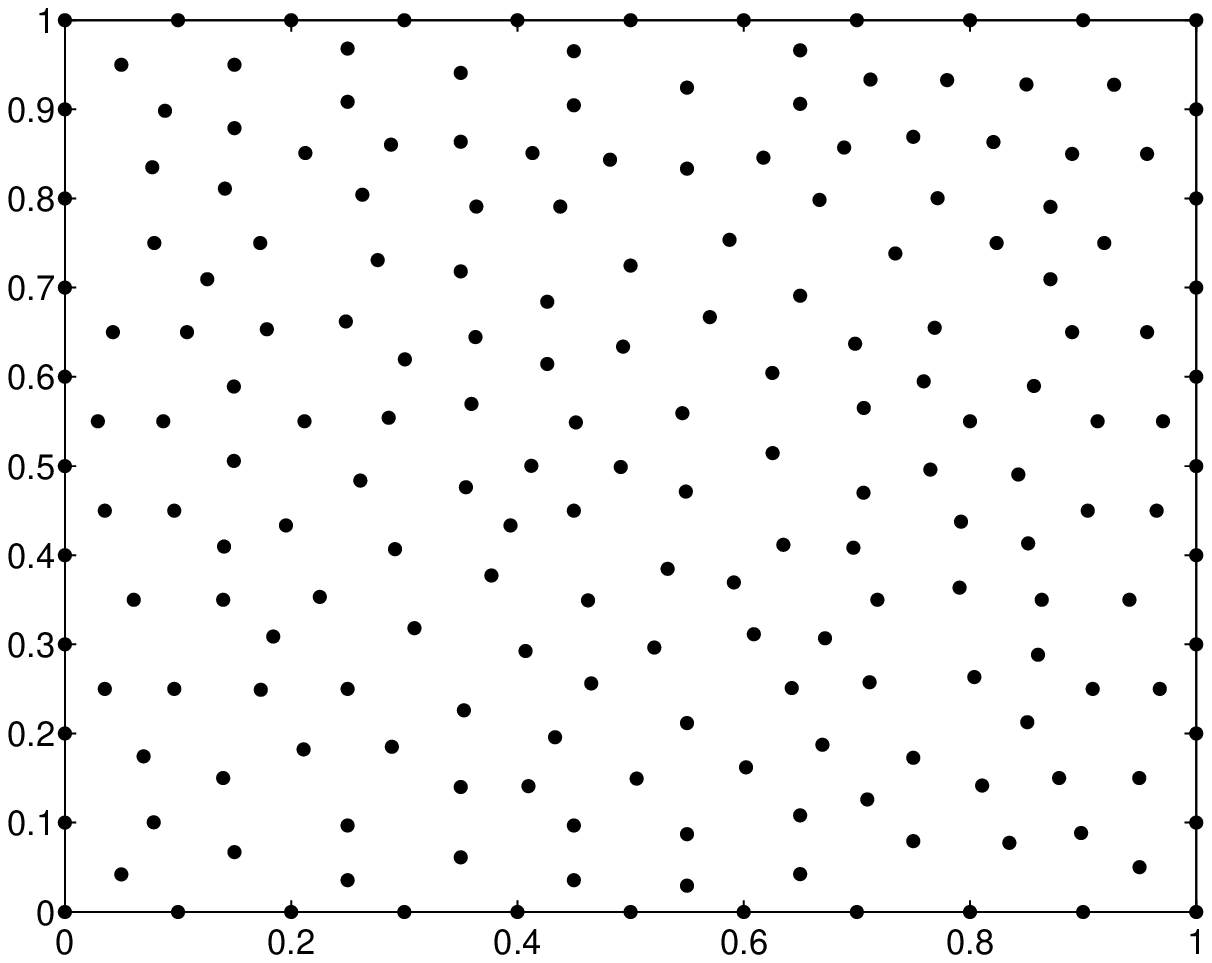}\hspace{-0.1cm}
\includegraphics[width=.33\linewidth]{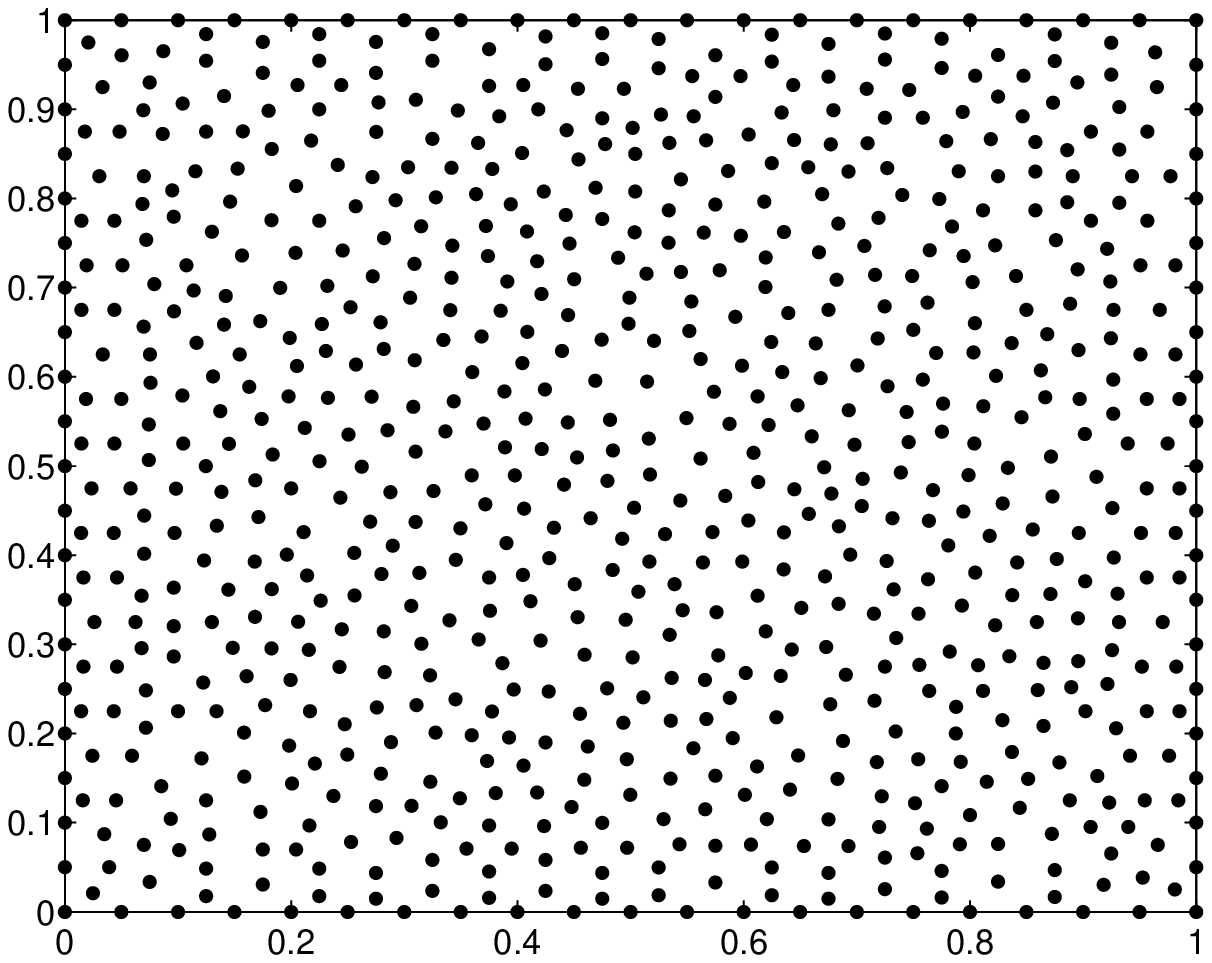}\hspace{0cm}
\includegraphics[width=.33\linewidth]{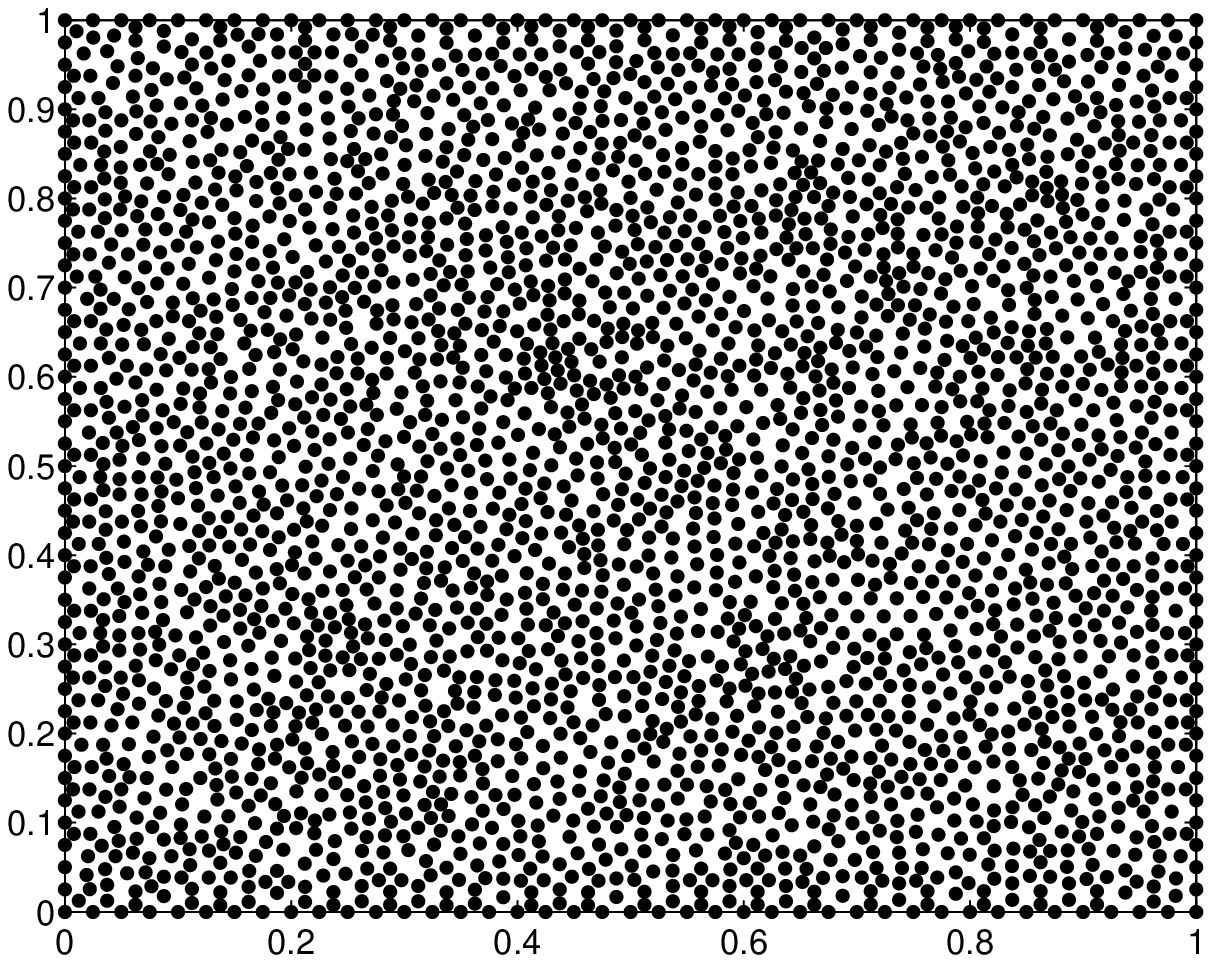}\caption{Three of the
five sets of nodes used in our numerical experiments with $N=202$ (left),
$N=777$ (middle), and $N=2991$ (right) nodes.}%
\label{fig:set-of-nodes}%
\end{figure}

\label{Sec:5}To test the accuracy of approximation of the bivariate
Shepard-Bernoulli operator we apply it to different sets of nodes in the rectangle $R=\left[  0,1\right]
\times\left[  0,1\right]  $ (see Figure
\ref{fig:set-of-nodes}) and $10$ test functions (see Figure
\ref{fig:test-functions}) generally used in the multivariate interpolation
of large sets of scattered data \cite{RenCli,RenBro}. In the following we
report the results of some of these experiments.

\begin{figure}[ptb]
{\small \centering \parbox{.15\linewidth}{\centering
\includegraphics[width=\linewidth]{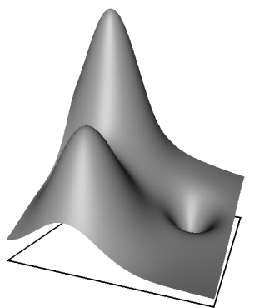}\par
$f_1$\\[1ex]
\includegraphics[width=\linewidth]{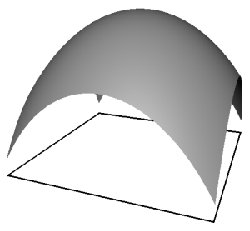}\par
$f_6$}\hfill\parbox{.15\linewidth}{\centering
\includegraphics[width=\linewidth]{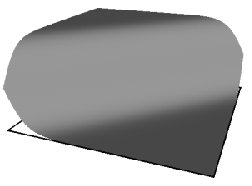}\par
$f_2$\\[1ex]
\includegraphics[width=\linewidth]{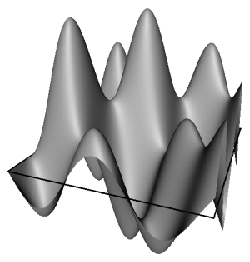}\par
$f_7$}\hfill\parbox{.15\linewidth}{\centering
\includegraphics[width=\linewidth]{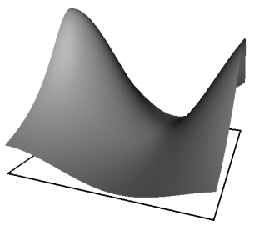}\par
$f_3$\\[1ex]
\includegraphics[width=\linewidth]{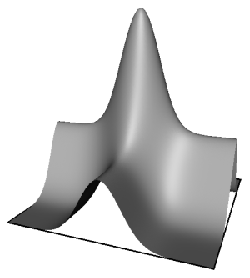}\par
$f_8$}\hfill\parbox{.15\linewidth}{\centering
\includegraphics[width=\linewidth]{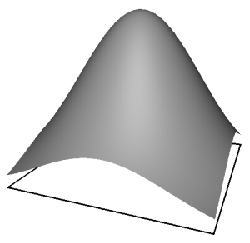}\par
$f_4$\\[1ex]
\includegraphics[width=\linewidth]{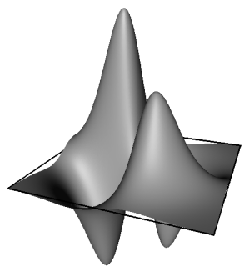}\par
$f_{9}$}\hfill\parbox{.15\linewidth}{\centering
\includegraphics[width=\linewidth]{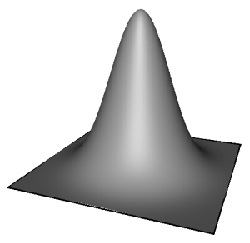}\par
$f_5$\\[1ex]
\includegraphics[width=\linewidth]{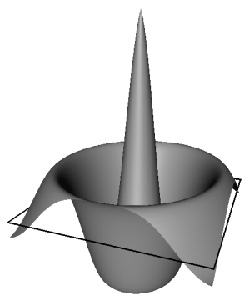}\par
$f_{10}$}}\caption{Test functions used in our numerical experiments. The
definitions of the test functions can be found in~\cite{Ren2}.}%
\label{fig:test-functions}%
\end{figure}

\subsection{Error of approximation when derivative data are given}

In a first series of experiments, we consider the case in which at each node
$V_{i}$ function evaluations and derivative data up to the order $2$ are
given. In this case, for each function $f_{i},i=1,\ldots,10$ we compare the
numerical results obtained by applying the approximation operator $S_{B_{3}%
}\left[  f\right]  $ with those obtained by the local version of the famous
Shepard-Taylor operator \cite{Far,Zup}%
\begin{equation}
S_{T_{2}}\left[  f\right]  \left(  x,y\right)  =\sum\limits_{i=1}%
^{N}\widetilde{W}_{\mu,i}\left(  x,y\right)  T_{2}\left[  f,V_{i}\right]
\left(  x,y\right)  \label{shepard-taylor}%
\end{equation}
which uses the same data. We report the results for the first four functions
in Figure \ref{fig:BSHEPvsTSHEP}, where we show the maximum interpolation
errors, computed for the parameter value $N_{w}=9$. The remaining six
functions have a similar behaviour and for this reason we omit them. Numerical
results show that the operator $S_{B_{3}}$ improves the accuracy of the
operator $S_{T_{2}}$.

\begin{figure}[ptb]
\centering
\includegraphics[width=.50\linewidth]{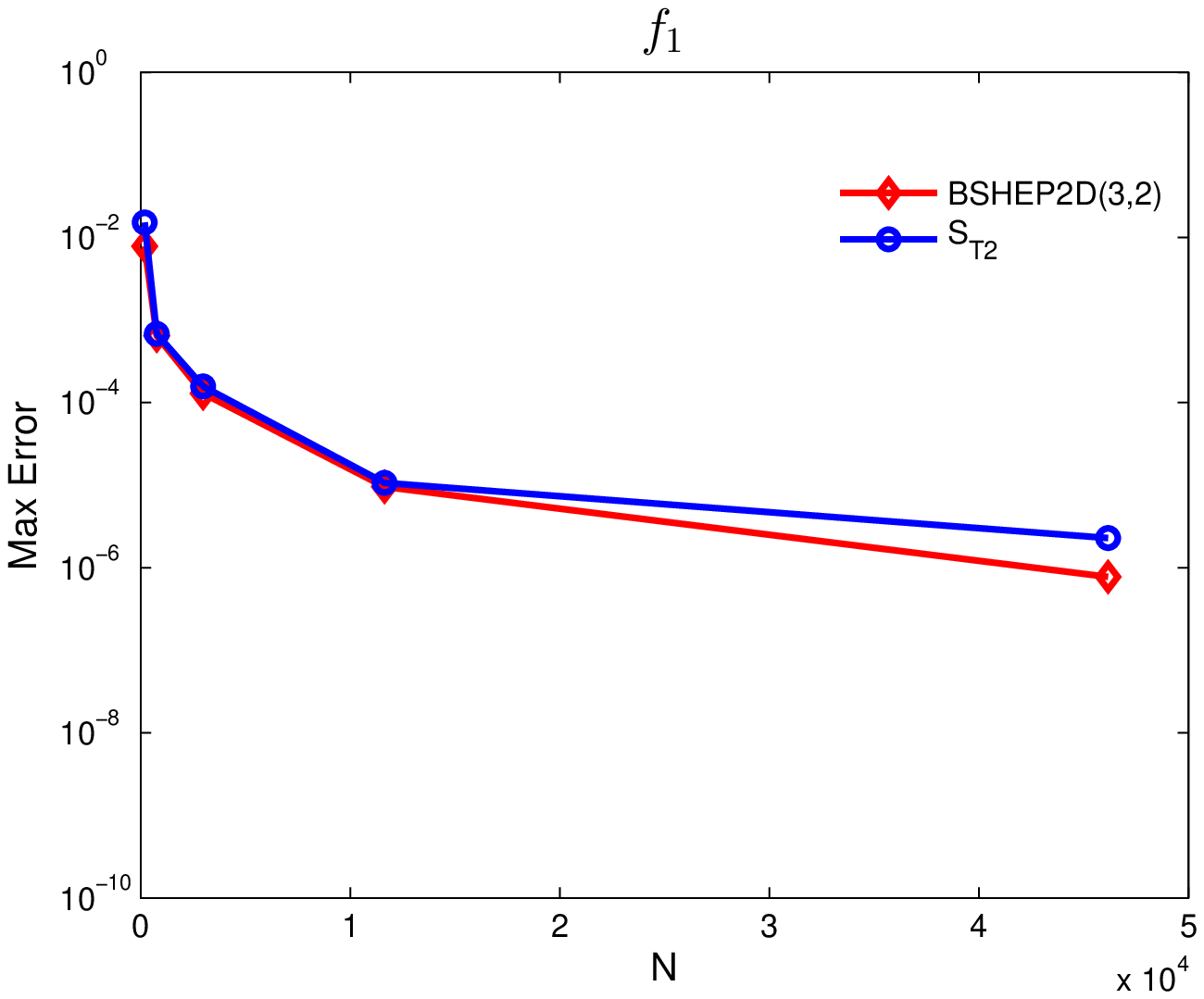}\hfill
\includegraphics[width=.50\linewidth]{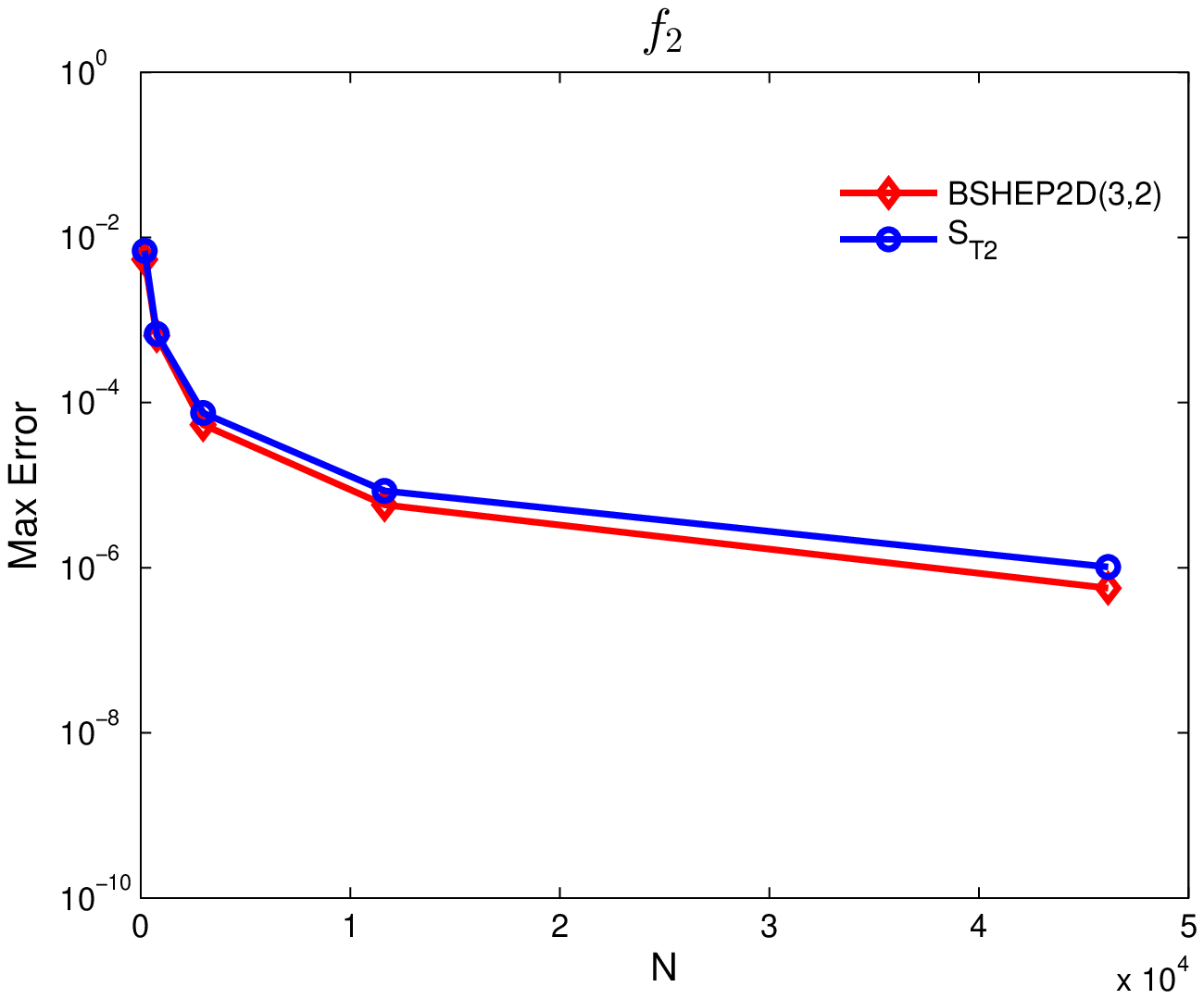}\hfill
\includegraphics[width=.50\linewidth]{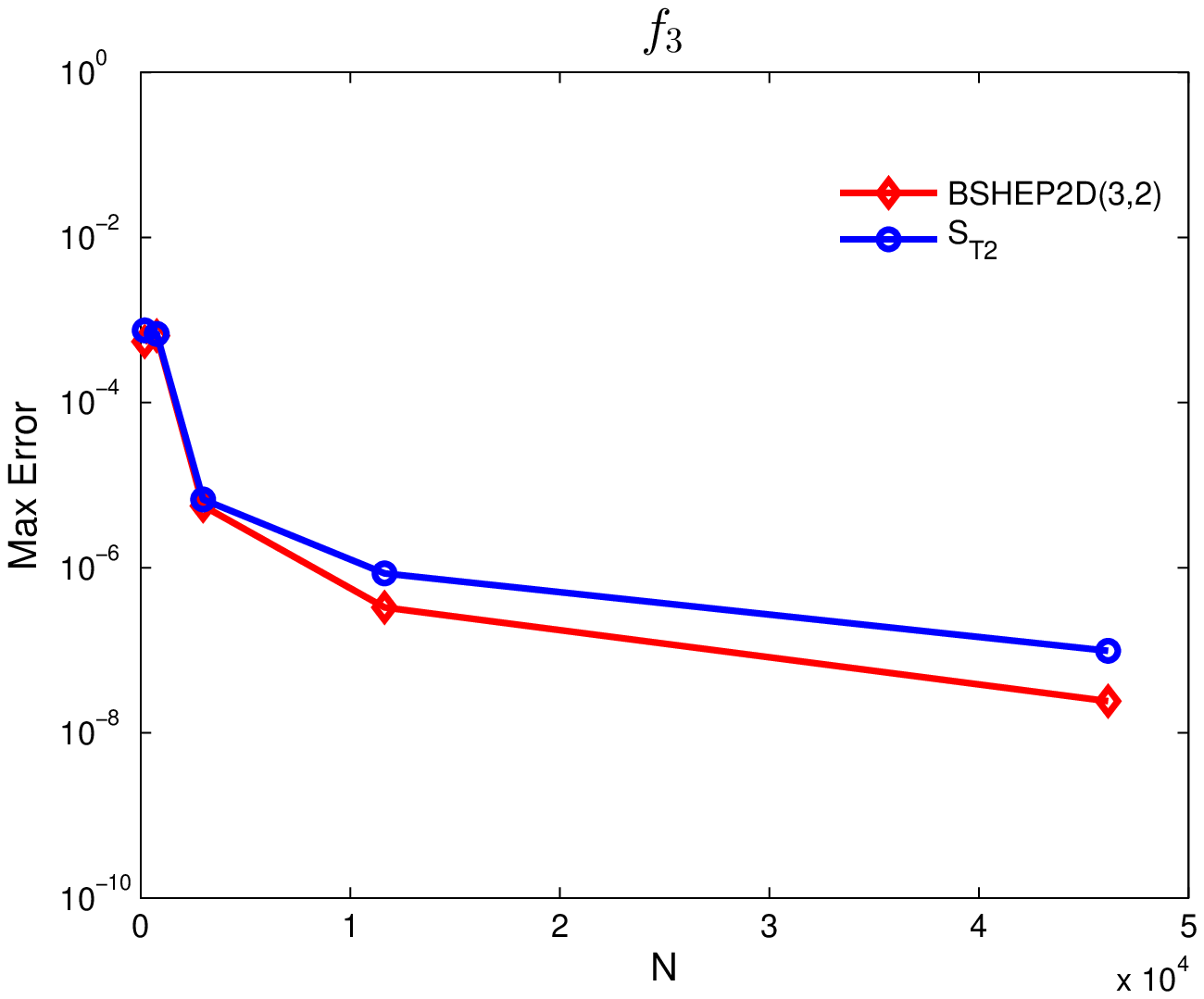}\hfill
\includegraphics[width=.50\linewidth]{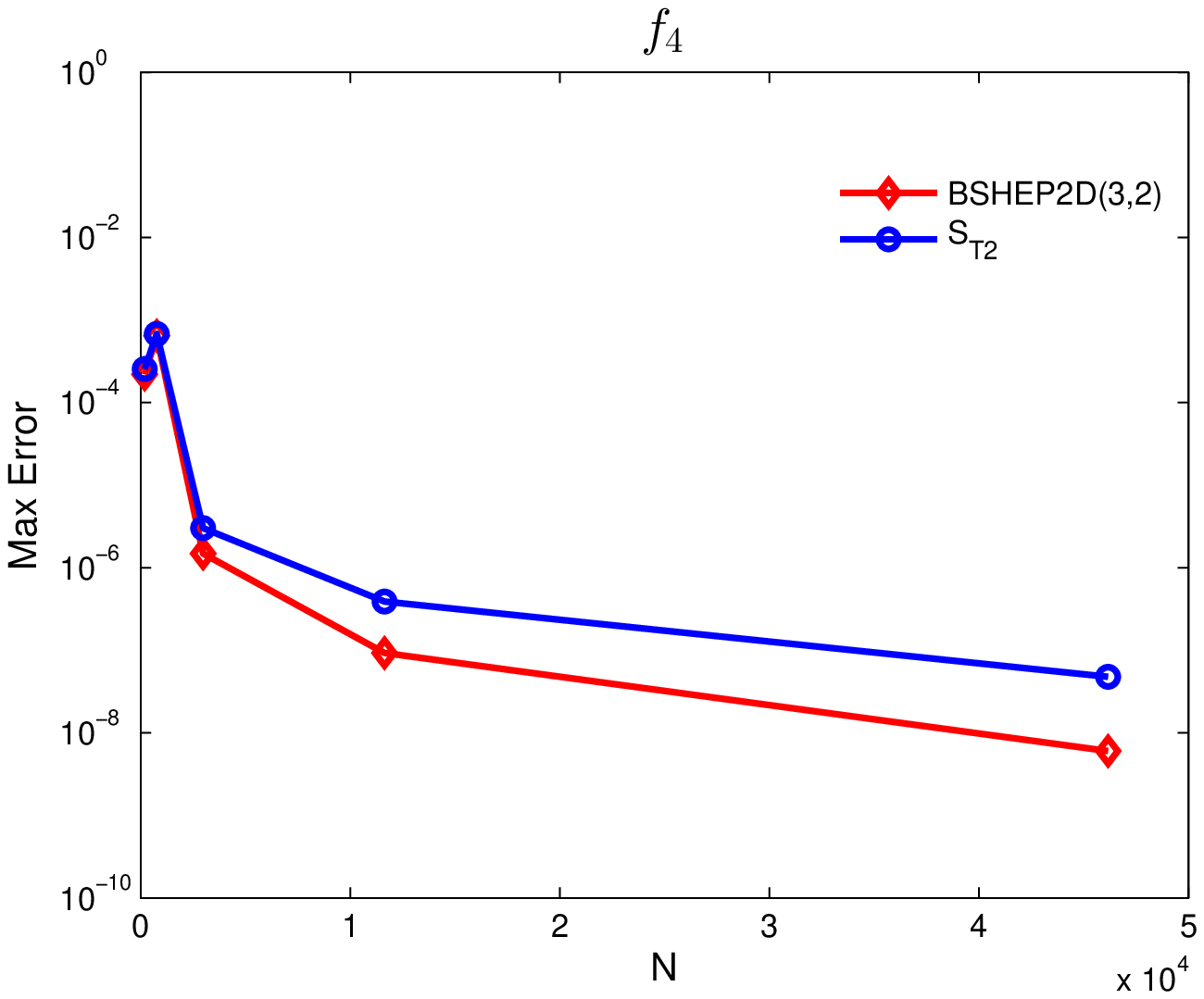}\hfill
\caption{Comparison between the maximum approximation error using $S_{T_{2}}$
operator and BSHEP2D(3,2) operators when function evaluations and derivative
data up to the order $2$ are given.}%
\label{fig:BSHEPvsTSHEP}%
\end{figure}

\subsection{Error of approximation when only function evaluations are given}

In a second series of experiments, we consider the case in which at each node
$V_{i}$ only function evaluations are given. In this case, the second order
derivatives
\begin{equation}
\frac{\partial f}{\partial x},\frac{\partial f}{\partial y},\frac{1}{2}%
\frac{\partial^{2}f}{\partial x^{2}},\frac{\partial^{2}f}{\partial x\partial
y},\frac{1}{2}\frac{\partial^{2}f}{\partial y^{2}} \label{derivativedata}%
\end{equation}
are usually replaced by the coefficients $a_{10},a_{01},a_{20},a_{11},a_{02}$
of the quadratic polynomial
\[
Q_{i}\left(  x,y\right)  =f\left(  V_{i}\right)  +\sum_{\substack{r+s=1\\r\geq
0,s\geq0}}^{2}a_{rs}\left(  x-x_{i}\right)  ^{r}\left(  y-y_{i}\right)  ^{s}%
\]
which fits data values $\left(  V_{k},f\left(  V_{k}\right)  \right)
,k=1,\ldots,N$ on a set of nearby nodes in a weighted least-square sense, as
in the definition of operator QSHEP2D \cite{Ren2}. The procedure for computing
these coefficients is well detailed in\ \cite{RenCli} and it is based on the choice
of another radius of influence about node $V_{i}$, $R_{q}$, which varies with
$i$ and is taken to be just large enough to include $N_{q}$ nodes in $B\left(
\boldsymbol{x}_{i},R_{q_{i}}\right)  $. At the same time the derivative data
(\ref{derivativedata}) at $V_{i},i=1,\ldots,N$ can be replaced by the
coefficients $b_{10},b_{01},b_{20},b_{11},b_{02}$ of the cubic polynomial
\[
C_{i}\left(  x,y\right)  =f\left(  V_{i}\right)  +\sum_{\substack{r+s=1\\r\geq
0,s\geq0}}^{3}b_{rs}\left(  x-x_{i}\right)  ^{r}\left(  y-y_{i}\right)  ^{s}%
\]
which fits the data values $\left(  V_{k},f\left(  V_{k}\right)  \right)
,k=1,\ldots,N$ on a set of nearby nodes in a weighted least-square sense, as
in the definition of operator CSHEP2D in \cite{Ren4}. In the following we
denote by BSHEP2D(3,2) the Shepard-Bernoulli operator obtained by substituting
the partial derivatives in $S_{B_{3}}\left[  f\right]  $ with linear
combinations of $a_{10},a_{01},a_{20},a_{11},a_{02}$ and by BSHEP2D(3,3) the
Shepard-Bernoulli operator obtained by substituting the partial derivatives in
$S_{B_{3}}\left[  f\right]  $ with linear combinations of $b_{10}%
,b_{01},b_{20},b_{11},b_{02}$. Therefore the operator BSHEP2D(3,2) has degree
of exactness $2$ as the operator QSHEP2D, while the operator BSHEP2D(3,3) has
degree of exactness $3$ as the operator CSHEP2D. We report the results for the
first four functions in Figure \ref{fig:BSHEPvsQSHEP}, where we show the
maximum interpolation errors, computed for the parameter value $N_{w}=9$ and
$N_{q}=13$ for the operator QSHEP2D and $N_{q}=17$ when we replace the derivative data by using the cubic polynomial $C_i(x,y)$.
The remaining six functions have a similar behaviour and for this reason we
omit them. Numerical results show that the operator $S_{B_{3}}$ improves the
accuracy of the operator QSHEP2D and is comparable with the operator CSHEP2D.

\begin{figure}[ptb]
\centering
\includegraphics[width=.50\linewidth]{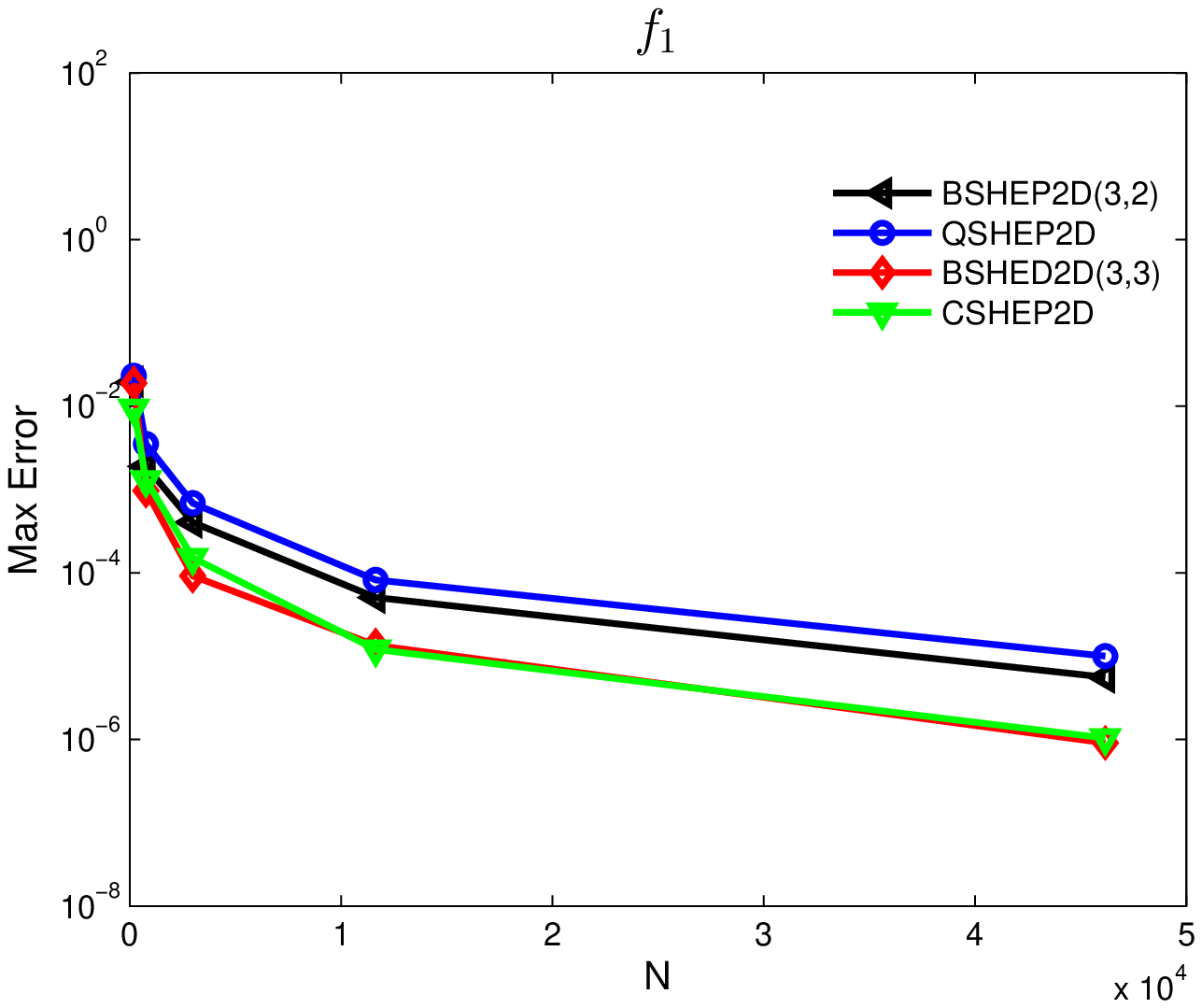}\hfill
\includegraphics[width=.50\linewidth]{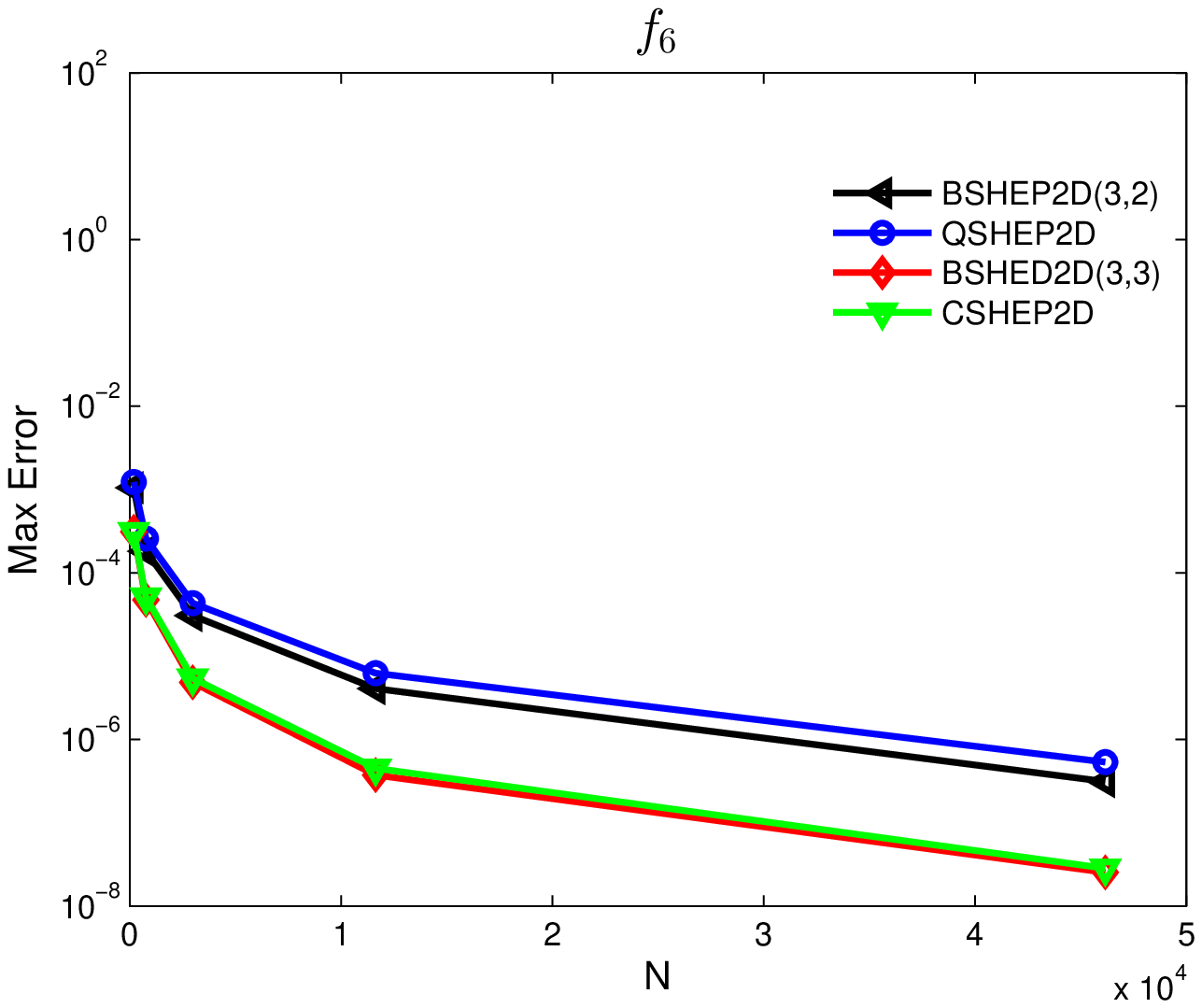}\hfill
\includegraphics[width=.50\linewidth]{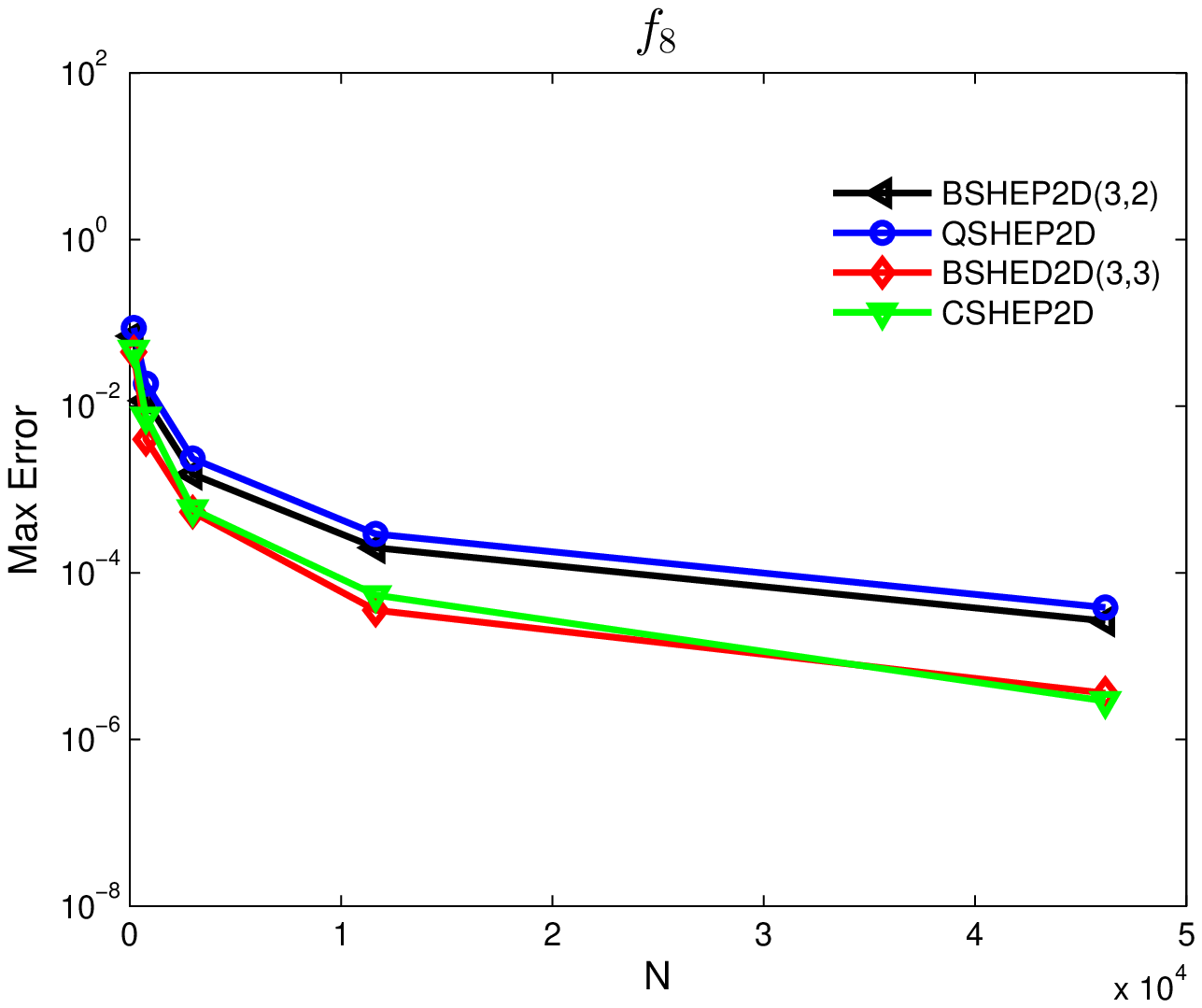}\hfill
\includegraphics[width=.50\linewidth]{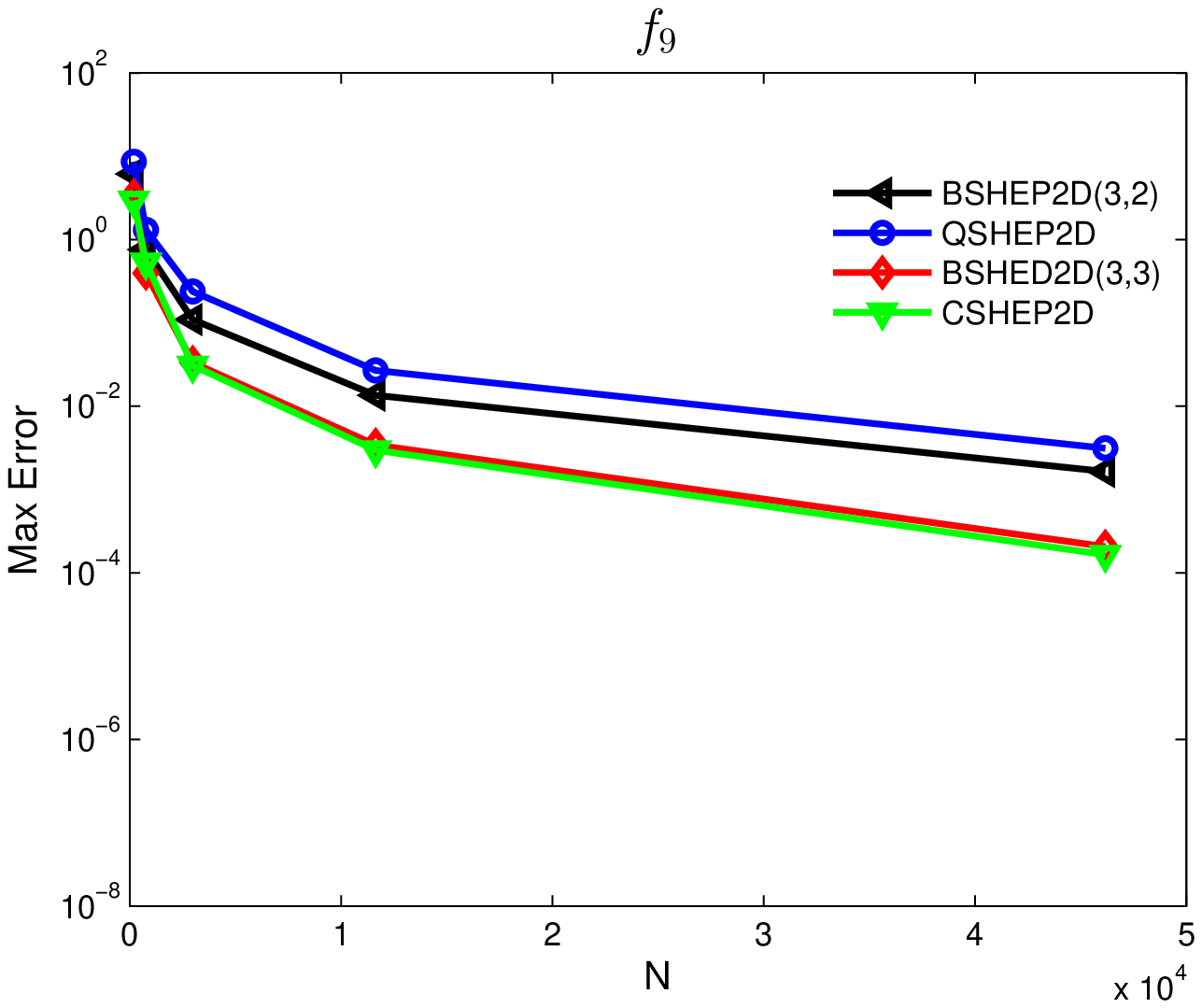}\hfill\caption{Comparison
between the maximum approximation error using QSHEP2D and CSHEP2D operators
and BSHEP2D(3,2) and BSHEP2D(3,3) operator when only function evaluations are
given.}\label{fig:BSHEPvsQSHEP}
\end{figure}

With regard to the computational cost we note that the
point-triangle associations which reduce the error of the three point
interpolation polynomials (\ref{gspoltayexp}) involves an additional cost of
$O\left(  N\right)  $ calculations that do not modify the computational cost
of operator QSHEP2D which is $O\left(  N\right)  $ for uniform
distributions of nodes and $O\left(  N^{2}\right)  $ in worst cases
\cite{RenBro}. On the other hand, local basis (\ref{Wtbasisfunctions}) used to
define our operators, containing a considerably lower number of nodes compared
to the choices recommended by Renka, involve a better localization of the
combined operator.

\section{Conclusions}

\label{Sec:6}In this paper we propose a new definition of the bivariate
Shepard-Bernoulli operators which avoids the drawbacks of Catinas extension
\cite{Cat}. These new interpolation operators are realized by using local
support basis functions introduced in \cite{FraNie} instead of classical
Shepard basis functions and the bivariate three point extension \cite{CosDel2}%
\ of the generalized Taylor polynomial introduced\ by F. Costabile in
\cite{Cos}.\ Their definition requires the association, to each sample point,
of a triangle with a vertex in it and other two vertices in its neighborhood. The proposed
point-triangle association is carried out to reduce the error of the three
point interpolation polynomial. As a consequence, the resulting operator not
only inherits interpolation conditions that each three point local interpolation polynomial satisfies at the referring vertex and increases by $1$ the degree of exactness of the
Shepard-Taylor operator \cite{Far} which uses the same data, but also improves
its accuracy. In this sense, the Shepard-Bernoulli operators belong to a recently introduced class of operators for enhancing the approximation order of Shepard operators by using supplementary derivative data \cite{CaiDelDit,CosDelDit1,CosDelDit2}. For the general problem of the enhancement of the algebraic precision of linear operators of approximation see the papers  \cite{Han,Guessab,Lam} and the references therein.
Moreover, when applied to the scattered data interpolation
problem, the Shepard-Bernoulli operator $S_{B_{3}}$ improves the
accuracy of the operator QSHEP2D and is comparable with the operator CSHEP2D by Renka \cite{Ren2,Ren4}. Finally, the
quadratic triangular finite element $P_{2}^{\left[  \Delta_{2}\left(
V_{1},V_{2},V_{3}\right)  ;V_{3}\right]  }[f]\left(  \boldsymbol{x}\right)  $
can be used to improve the accuracy of approximation of the triangular Shepard
method \cite{Little}.

\end{document}